\documentclass[conference]{IEEEtran}
\IEEEoverridecommandlockouts
\usepackage{cite}
\usepackage{amsmath, amsfonts, amsthm}
\usepackage{graphicx}
\usepackage{subcaption}
\usepackage{xcolor}
\usepackage{enumerate}
\usepackage{comment}
\usepackage[]{algorithm2e}

\usepackage{geometry}
\newgeometry{top=1in,bottom=0.75in,right=0.75in,left=0.75in}

\theoremstyle{plain}
\newtheorem{proposition}{Proposition}
\newtheorem{lemma}{Lemma}
\theoremstyle{definition}

\newtheorem{assumption}{Assumption}
\newtheorem{claim}{Claim}
\theoremstyle{remark}
\newtheorem{remarkTheo}{Remark}
\newenvironment{remark}[1][]{\begin{remarkTheo}}{\hfill$\diamond$\end{remarkTheo}}

\def \ba{\begin{array}}
\def \ea{\end{array}}
\def \be{\begin{equation}}
\def \ee{\end{equation}}

\def \bb{\mathbb}

\def \mc{\mathcal}

\def \ms{\mathsf}

\def \eig{\mathrm{eig}}

\def \train{\text{train}}
\def \test{\text{test}}

\def \T{\mathtt{T}}

\begin{document}
%
\title{Sensor Fault Detection and Isolation in Autonomous Nonlinear Systems Using Neural Network-Based Observers}

\author{John Cao$^*$ \and M. Umar B. Niazi$^{\dagger}$ \and Matthieu Barreau$^*$ \and Karl Henrik Johansson$^*$
\thanks{$*$~Division of Decision and Control Systems, Digital Futures, KTH Royal Institute of Technology, SE-100 44 Stockholm, Sweden. Emails: \texttt{johncao@kth.se}, \texttt{barreau@kth.se}, \texttt{kallej@kth.se}}
\thanks{$\dagger$~Laboratory for Information \& Decision Systems, Department of Electrical Engineering and Computer Science, Massachusetts Institute of Technology, Cambridge, MA 02139, USA. Email: \texttt{niazi@mit.edu}}
\thanks{This work is supported by the Swedish Research Council and the Knut and Alice Wallenberg Foundation, Sweden. It has also received funding from the European Union's Horizon Research and Innovation Programme under Marie Sk\l{}odowska-Curie grant agreement No. 101062523.}
}

\maketitle

\begin{abstract}
This paper presents a novel observer-based approach to detect and isolate faulty sensors in nonlinear systems. The proposed sensor fault detection and isolation (s-FDI) method applies to a general class of nonlinear systems. Our focus is on s-FDI for two types of faults: complete failure and sensor degradation. The key aspect of this approach lies in the utilization of a neural network-based Kazantzis-Kravaris/Luenberger (KKL) observer. The neural network is trained to learn the dynamics of the observer, enabling accurate output predictions of the system. Sensor faults are detected by comparing the actual output measurements with the predicted values. If the difference surpasses a theoretical threshold, a sensor fault is detected. To identify and isolate which sensor is faulty, we compare the numerical difference of each sensor meassurement with an empirically derived threshold. We derive both theoretical and empirical thresholds for detection and isolation, respectively. Notably, the proposed approach is robust to measurement noise and system uncertainties. Its effectiveness is demonstrated through numerical simulations of sensor faults in a network of Kuramoto oscillators.
\end{abstract}


\section{Introduction}

Sensor fault detection and isolation (s-FDI) plays a pivotal role in ensuring the safe and efficient operation of numerous industrial processes. We address two distinct types of sensor faults: complete failure and sensor degradation. Complete failure occurs when a sensor becomes entirely nonfunctional, often due to mechanical breakdown or a similar issue. In contrast, sensor degradation results in a gradual decline in the sensor's measurement accuracy. When left undetected, these sensor faults can lead to disruptive consequences for the system. Effective s-FDI methods serve as a proactive solution, allowing system operators to detect sensor faults early, localize the faulty sensor, and take corrective action before they escalate into issues that could result in costly damage or downtime. 

Fault detection and isolation methods can generally be categorized into two main groups: hardware redundancy and analytical redundancy, as discussed in  \cite{VENKATASUBRAMANIAN2003293}. Hardware redundancy methods entail using multiple sensors to acquire and cross-reference data from processes. While effective, this approach does come with the drawback of increased costs for acquiring and maintaining additional hardware. Analytical redundancy methods offer an alternative solution. These methods rely on the principle of residual generation, where the residual represents the variance between a predicted system output and the actual measurement. In the absence of faults, this residual remains close to zero. However, as soon as faults occur, it significantly exceeds a predetermined threshold, facilitating the detection of anomalies without the need for additional hardware.

Historically, s-FDI methods based on analytical redundancy have primarily been model-based, relying on explicit mathematical models of the system under consideration, as noted in \cite{comp_methods}. This model-based approach was initially developed for linear systems in the 1970s, exemplified by pioneering work such as \cite{beard1971failure}, which demonstrated the feasibility of designing filters for detecting and localizing faults within observable system dynamics. Subsequent refinements and enhancements, as seen in \cite{jones1973failed}, resulted in the renowned ``Beard-Jones Fault Detection filter." Further extensions and practical applications of these methods can be found in works such as \cite{massoumnia1986a, massoumnia1986b, massoumnia1989failure, white1987detection, min1987detection}.

Parallel to these developments, the framework of observer-based fault detection schemes emerged for linear systems, with an early reference being \cite{clark1975}. Over time, this approach has gained recognition as one of the most successful methods for s-FDI, leading to diverse research directions. For instance, the application of sliding-mode observers for s-FDI by \cite{EDWARDS2000541, YAN20071605,tan2003sliding} allowed explicit reconstruction of the sensor faults by manipulating the output injection error. Nonlinear unknown input observers have also been prevalent for s-FDI \cite{comp_methods}. However, in recent years, interval-based unknown input observers have gained significant prominence. These observers offer advantages because of relaxed assumptions on system inputs and nonlinearities \cite{zhu2022, xu2019conservatism, zhang2017event, su2020fault}.

Another principal approach to analytical redundancy-based s-FDI is by using data-driven methods, which have gained popularity over the past decade, thanks to significant advancements in deep learning. These methods do not require explicit system models, relying instead on sensor data to approximate the underlying dynamics and generate residuals, as highlighted in \cite{iqbal2019}. For instance, \cite{YANG201985} demonstrated s-FDI using long short-term memory neural networks (LSTM), where residuals were formed by comparing network predictions based on past time-series data with actual measurements. Similar neural network-based approaches have been proposed for various domains, including industrial manufacturing processes, power plants, and unmanned aerial vehicles, as indicated in \cite{iqbal2019, farahani2020, GUO2018155}.

Despite the rich diversity of existing s-FDI methods, it's important to acknowledge that, to the best of the authors' knowledge, these techniques are often challenging to implement in practice. Observer-based methods are constrained by their assumptions about specific system structures, while data-driven approaches require substantial sensor data, which can be both difficult and costly to obtain.

In this paper, we tackle these challenges by introducing a novel learning-based approach to s-FDI using neural network-based nonlinear Luenberger observers, specifically Kazantzis-Kravaris/Luenberger (KKL) observers \cite{kazantzis1998, andrieu2006existence}. KKL observers lift the original system to a higher-dimensional state space, where the system behaves linearly and is stable. The nonlinear transformation required by the lift is obtained by solving a certain partial differential equation. The observer's design is based on this transformed system. To estimate the system's state, an inverse transformation is applied, to obtain the estimate in the original state space. A notable advantage of KKL observers is their flexibility; they do not rely on specific triangular structure or normal form of the system but rather depend on relatively mild observability conditions, rendering them applicable to a broad class of nonlinear systems.

Our main contributions in this paper are 
\begin{enumerate}
    \item We develop a novel s-FDI algorithm using a neural network-based KKL observer with no assumptions on the structure of the system.
    \item We derive a theoretical threshold for sensor fault detection based on the residual, and also devise an empirical method to obtain a threshold for sensor fault isolation.
    \item Through simulations, we show that the proposed method is able to effectively detect and isolate sensor faults under a variety of circumstances while remaining robust and functional under the influence of model uncertainties and measurement noise.
\end{enumerate}

The outline of this paper is as follows. Section~\ref{problem_formulation} formulates the s-FDI problem under some mild assumptions. Section~\ref{FDI_learning} presents the proposed s-FDI methhod. Numerical results are provided in Section~\ref{simluation_results} where the method is tested in simulations on a variety of fault cases. Lastly, Section~\ref{conclusion} concludes the paper.

\section{Problem Formulation}\label{problem_formulation}

We consider a nonlinear system 
\begin{subequations}
\label{system}
\begin{align}
    \dot{x}(t) &= f(x(t)) + w(t)
    \label{system_state} \\
    y(t) &= \phi(t) \left[h(x(t)) + v(t) + \zeta(t)\right] 
    \label{system_output}
\end{align}
\end{subequations}
where $x(t)\in\mathcal{X}\subset\mathbb{R}^{n_x}$ is the state, $y(t)\in\mathbb{R}^{n_y}$ is the output, $f:\bb{R}^{n_x}\rightarrow\bb{R}^{n_x}$ and $h:\bb{R}^{n_x}\rightarrow\bb{R}^{n_y}$ are smooth maps, and $w(t),v(t)$ are process and measurement noises, respectively.
The measured output \eqref{system_output} might be affected by sensor faults $\phi(t)$ and $\zeta(t)$, where
\[
\phi(t) = \begin{bmatrix} \phi_{1}(t) & \dots & \phi_{n_y}(t) \end{bmatrix}^\T \in \{0,1\}^{n_{y}}
\]
models the complete failure of sensor~$i$ when $\phi_i(t)=0$, and 
\[
\zeta(t) = \begin{bmatrix} \zeta_{1}(t) & \dots & \zeta_{n_y}(t) \end{bmatrix}^\T \in \mathbb{R}^{n_{y}}
\]
models the degradation of sensor~$i$, affecting its measurement accuracy, when $\zeta(t)\neq 0$.

Note that a system with fault-free sensors at time $t$ will have $\phi(t) = 1_{n_y}$ and $\zeta(t) = 0_{n_y}$. Any other value of these signals represents a type of fault, such as sensor bias or degradation when $\zeta_i(t) \neq 0$, for $t\geq T$, or sensor failure when $\phi_i(t) = 0$, for $t\geq T$, where $i\in\{1,\dots,n_y\}$ and $T\in\bb{R}_{>0}$ is a time at which fault occurs. 

\begin{assumption}
    \label{assume-1}
    We assume the following regarding the process and measurement noise:
    \begin{enumerate}[(i)]
        \item Essential boundedness: $\|w\|_{L^\infty}\leq \Bar{w}$ and $\|v\|_{L^\infty}\leq \Bar{v}$, where $\Bar{w},\Bar{v}>0$ are known and $\|\cdot\|_{L^\infty}$ denotes the essential supremum norm.
        \item Bounded effect: There exists a class $\mc{K}_\infty$ function $\psi:\bb{R}\rightarrow\bb{R}_{\geq 0}$ such that
        \[
        \|x(t;x_0,w)-x(t;x_0,0)\| \leq \psi(\Bar{w})
        \]
        where $x(t;x_0,w)$ denotes the state trajectory of \eqref{system_state} initialized at $x(0)=x_0$ and driven by the noise $w(t)$, and $x(t;x_0,0)$ denotes the state trajectory in the absence of noise.
    \end{enumerate}
\end{assumption}

\begin{remark}
    Assumption~\ref{assume-1}(i) is a standard assumption in the robust state estimation literature. It says that the signals $w(t)$ and $v(t)$ almost always remain bounded, and the instances at which $w(t)$ or $v(t)$ go unbounded are of zero Lebesgue measure. On the other hand, Assumption~\ref{assume-1}(ii) requires that the model $f$ without a process noise does a good job of describing the noisy system \eqref{system_state}. Such guarantees are usually provided in the system identification literature, and the reader is referred to \cite{milanese1996,mania2022} for more details.
\end{remark}

Given a system \eqref{system}, a neural network-based Kazantzis-Kravaris/Luenberger (NN-KKL) observer \cite{niazi2022} is given by
\begin{subequations}
\begin{align}
    &\dot{\hat{z}}(t) = A\hat{z}(t) + By(t) \label{observer_state} \\
    & \hat{x}(t) = \hat{\mathcal{T}}_\eta^*(\hat{z}(t)) \label{observer_estimate} \\
    &\hat{y}(t) = h(\hat{x}(t)) \label{observer_output}
\end{align}
\label{observer}%
\end{subequations}
which takes output measurements $y(t)$ as input and gives the estimated state $\hat{x}(t)$ and predicted output $\hat{y}(t)$ as outputs. In \eqref{observer_state}, $A \in \mathbb{R}^{n_{z} \times n_{z}}$ is a Hurwitz matrix and $B \in \mathbb{R}^{n_{z} \times n_{y}}$ is such that the pair $(A,B)$ is controllable. 
The observer state $\hat{z}(t)\in\mathbb{R}^{n_z}$ follows a nonlinear transformation $\hat{z} = \hat{\mathcal{T}}_\theta(\hat{x})$, where $\hat{\mathcal{T}}_\theta$ is a neural network approximation of $\mathcal{T} : \mathcal{X} \rightarrow \mathbb{R}^{n_{z}}$, which is an injective transformation of the state space to a higher dimension $n_z=n_y(2n_x+1)$. On the other hand, in \eqref{observer_estimate}, $\hat{\mathcal{T}}^*_\eta$ is a neural network approximation of $\mathcal{T}^* : \mathbb{R}^{n_{z}} \rightarrow \mathcal{X}$, which is the left inverse of $\mathcal{T}$. The maps $\hat{\mc{T}}_\theta$ and $\hat{\mathcal{T}}^*_\eta$ are neural networks with parameters $\theta$ and $\eta$, respectively. Note that the observer is trained by considering fault-free sensors. We refer to \cite{niazi2022} for details on the design of NN-KKL observers.

\begin{assumption}
    \label{assume-2}
    We assume the following regarding nonlinear transformations $\mc{T}, \mc{T}^*$ and their corresponding neural network approximations $\hat{\mc{T}}_\theta, \hat{\mc{T}}_\eta^*$:
    \begin{enumerate}[(i)]
        \item Lipschitzness: $\hat{\mc{T}}_\theta$ and $\hat{\mc{T}}_\eta^*$ are Lipschitz continuous over $\mc{X}$ and $\mc{Z}$, respectively, where $\mc{Z}\supseteq \mc{T}(\mc{X})$. In particular, for every $x,\hat{x}\in\mc{X}$, there exists $\ell_\theta\in\bb{R}_{>0}$ such that
        \[
        \|\hat{\mc{T}}_\theta(x)-\hat{\mc{T}}_\theta(\hat{x})\| \leq \ell_\theta \|x-\hat{x}\|
        \]
        and, for every $z,\hat{z}\in\mc{Z}$, there exists $\ell_\eta\in\bb{R}_{>0}$ such that
        \[
        \|\hat{\mc{T}}_\eta^*(z)-\hat{\mc{T}}_\eta^*(\hat{z})\| \leq \ell_\eta \|z-\hat{z}\|.
        \]
        \item Uniform injectivity: There exists a class $\mc{K}$ functions $\rho,\rho^*$ such that, for every $x,\hat{x}\in\mc{X}$,
        \[
        \|x-\hat{x}\| \leq \rho(\|\mc{T}(x)-\mc{T}(\hat{x})\|)
        \]
        and, for every $z,\hat{z}\in\mc{Z}$,
        \[
        \|z-\hat{z}\| \leq \rho^*(\|\mc{T}^*(z) - \mc{T}^*(\hat{z})\|).
        \]
    \end{enumerate}
\end{assumption}

\begin{remark}
    Assumption~\ref{assume-2}(i) is satisfied if the activation function of neural networks is chosen to be Lipschitz continuous. Among others, one example is a neural network with ReLU activation function, which is Lipschitz with the Lipschitz constant equal to 1. On the other hand, Assumption~\ref{assume-2}(ii) is needed for the existence of a KKL observer \cite{andrieu2006existence}. Notice that uniformly injective $\mc{T}$ implies the uniform injectivity of its inverse $\mc{T}^*$. It is important to remark that Assumption~\ref{assume-2}(i) and (ii) imply boundedness of approximation errors, i.e., 
    \[
        \|\mc{T}(x)-\hat{\mc{T}}_\theta(x)\|<\infty, \; \forall x\in\mc{X}
    \]
    and 
    \[
        \|\mc{T}^*(z)-\hat{\mc{T}}_\eta^*(z)\|<\infty, \; \forall z\in\mc{Z}
    \]
    where $\bb{R}^{n_z}\supset\mc{Z}\supseteq \mc{T}(\mc{X})$.
    This holds because 
    \[
        \|\mc{T}(x)-\hat{\mc{T}}_\theta(x)\|\leq\|\mc{T}(x)\|+\|\hat{\mc{T}}_\theta(x)\|\leq \rho^*(\|x\|) + \ell_\theta\|x\|
    \]
    and $\mc{X}\subset\bb{R}^{n_x}$ is a compact set.   
    Further improvement on this bound can be achieved using the universal approximation property of neural networks given that an ample amount of data is generated and an appropriate network architecture is chosen.
\end{remark}

Observer-based s-FDI of nonlinear systems relies on the design of observers that are accurate in state estimation and output prediction. However, for general nonlinear systems, designing observers is a challenging task. Therefore, we consider an NN-KKL observer \eqref{observer} proposed in \cite{niazi2022}, which is designed using a physics-informed learning approach. The focus of the present paper is on developing an s-FDI method for general nonlinear systems using an NN-KKL observer~\eqref{observer}. Specifically, we address the following problems:
\begin{enumerate}[(P1)]
    \item Detection: How to detect whenever a fault occurs in one or more sensors?
    \item Isolation: How to identify a sensor whenever it becomes faulty?
\end{enumerate}



To solve the problems stated above, we define residuals to compare a measured output from the system with a predicted output from the NN-KKL observer. The residual corresponding to $i$th sensor is defined as
\begin{equation}
    \label{eq:residual-i}
    \arraycolsep=2pt
    \ba{ccl}
    r_i(t) & \doteq & |y_i(t)-\hat{y}_i(t)| \\
    & = & |\phi_i(t)[h_i(x(t))+v_i(t)+\zeta_i(t)] - h_i(\hat{x}(t))|.
    \ea
\end{equation}
Consider equally-distant discrete time samples $t_0,t_1,t_2,\dots$ with $\delta=t_k-t_{k-1}$ for $k\in\bb{N}$, then we define a differentiated residual of $i$th sensor as the absolute value of numerical differentiation of the residual, i.e.,
\begin{equation}
    \label{eq:inc_residual-i}
    \Tilde{r}_i(t_k) \doteq \left| \frac{r_i(t_k)-r_i(t_{k-1})}{\delta}\right|
\end{equation}
Let $\tau_i\in\bb{R}_{>0}$ denote a threshold for residual~$i$ such that, in the steady state, $r_i(t) \leq \tau_i$ when there are no faults. In addition, let $r_\Delta\in\bb{R}_{>0}$ denote a threshold for differentiated residuals such that, in the steady state, $\Tilde{r}_i(t)\leq r_\Delta$ when sensor~$i$ is not faulty.

\section{s-FDI using NN-KKL Observers}\label{FDI_learning}

In this section, we first derive theoretical bounds for the sensor residuals, which allow us to choose their thresholds. Then, we devise an empirical method to compute the threshold for differentiated residuals. Finally, we present the proposed s-FDI method.

\subsection{Upper bounds of the residuals}

We derive two upper bounds on the residuals. The first bound is straightforward but requires an additional stricter assumption that both $\hat{\mc{T}}_\theta$ and $\hat{\mc{T}}_\eta^*$ are contractions, i.e., their Lipschitz constants are less than 1. The second bound relies on Assumption~\ref{assume-1} and \ref{assume-2} and provides a bound that is more general.

\begin{proposition}
    \label{prop:res_contraction}
    Suppose the neural networks $\hat{\mc{T}}_\theta$ and $\hat{\mc{T}}_\eta^*$ are Lipschitz continuous with $\ell_\theta\in[0,1)$ and $\ell_\eta\in[0,1)$. Then, in a fault-less case,
    \be
        \label{eq:res_contraction}
        r_i(t) \leq \ell_{h_i} \frac{\ell_\eta \xi(x) + \xi^*(z)}{1-\ell_\theta\ell_\eta} + \Bar{v}
    \ee
    where $\ell_{h_i}$ is the Lipschitz constant of $h_i(x)$ for $x\in\mc{X}$, $\Bar{v}$ is given in Assumption~\ref{assume-1}(i), and 
    \begin{align}
        \xi(x) & \doteq \mc{T}(x) - \hat{\mc{T}}_\theta(x) \label{eq:xi} \\
        \xi^*(z) & \doteq \mc{T}^*(z) - \hat{\mc{T}}_\eta^*(z) \label{eq:xi*}
    \end{align}
    which are bounded for $x\in\mc{X}$ and $z\in\mc{Z}$ by Assumption~\ref{assume-2}(ii).
\end{proposition}
\begin{proof}
    See Appendix~\ref{appendix:proof_prop1}.
\end{proof}

\begin{remark}
    The assumption of Proposition~\ref{prop:res_contraction} that both Lipschitz constants are less than one can be satisfied by regularizing the weights of neural networks during training. For instance, an $l$-layer ReLU network with $W^{(1)},\dots,W^{(l)}$, which are the weight matrices of the neural network layers, has a Lipschitz constant $\ell=\|W^{(1)}\|\dots\|W^{(l)}\|$. By regularizing weights such that the maximum singular value of each weight matrix is less than one, it can be ensured that $\ell<1$. Nevertheless, we acknowledge that the upper bound \eqref{eq:res_contraction} is quite conservative because if $\ell_\theta\ell_\eta$ is close to one, the bound could become very large. In this case, such a bound may not be appropriate for choosing a threshold.
\end{remark}

Recall $z=\mc{T}(x)$ and $\hat{z}=\hat{\mc{T}}_\theta(x)$, where
\[
    \dot{z}(t) = Az(t) + Bh(\Bar{x}(t))
\]
with $\Bar{x}(t)\doteq x(t;x_0,0)$ the noise-free state trajectory, and the dynamics of $\hat{z}(t)$ is given in \eqref{observer_state}.

\begin{lemma} \label{lemma:zeta_bound}
Suppose the matrix $A$ is Hurwitz and diagonalizable with eigenvalue decomposition $A=V\Lambda V^{-1}$. Then, the error $\Tilde{z}(t):=z(t)-\hat{z}(t)$ satisfies
\[
    \ba{l}
    \|\Tilde{z}(t)\| \leq \|\Tilde{z}_0\| \kappa(V) e^{-c t} \\ [1em] 
    \qquad\qquad\displaystyle + \frac{\kappa(V)}{c} \|B\| (1-e^{-ct}) \big[\ell_h \psi(\Bar{w}) + \sqrt{n_y} \Bar{v} \big]
    \ea
\]
where $c$ is given in \eqref{eq:min_evalue}, $\psi$ in Assumption~\ref{assume-1}(ii), and $\ell_h$ is the Lipschitz constant of $h(x)$ for $x\in\mc{X}$.
\end{lemma}
\begin{proof}
See Appendix~\ref{appendix:proof_lemma2}.
\end{proof}

Using the above lemmas, we can obtain a general upper bound on the residuals under Assumption~\ref{assume-1} and \ref{assume-2}.

\begin{proposition}
    \label{prop:res_general}
    Suppose the matrix $A$ is diagonalizable with eigenvalue decomposition $A=V\Lambda V^{-1}$. Then, in a fault-less case,
    \be
        \label{eq:res_general}
        \ba{l}
        r_i(t) \leq \Bar{v} + \ell_{h_i} \bigg[\xi^*(z) + \ell_\eta \kappa(V) e^{-ct} \|\Tilde{z}_0\| \\ \hspace{1cm} + \frac{\kappa(V)}{c} \|B\|(1-e^{-ct})(\ell_h \psi(\Bar{w}) + \sqrt{n_y}\Bar{v}) \bigg]
        \ea
    \ee
    where $\Bar{v}$ is from Assumption~\ref{assume-1}(i), $\ell_{h_i}$ is the Lipschitz constant of $h_i(x)$ for $x\in\mc{X}$, $\xi^*(z)$ is defined in \eqref{eq:xi*}, $\kappa(V)$ is the condition number of $V$, $c$ is given in \eqref{eq:min_evalue}, $\Tilde{z}_0=\Tilde{z}(0)$, $\ell_h$ is the Lipschitz constant of $h(x)$ for $x\in\mc{X}$, and $\psi(\Bar{w})$ is from Assumption~\ref{assume-1}(ii).
\end{proposition}

\begin{proof}
    From \eqref{eq:prop_proof1} and $\eqref{eq:prop_proof2}$, we have
    \[
        r_i(t) \leq \Bar{v} + \ell_{h_i} \left[ \xi^*(z) + \ell_\eta \|z-\hat{z}\| \right].
    \]
    Therefore, by applying Lemma~\ref{lemma:zeta_bound}, we obtain \eqref{eq:res_general}.
\end{proof}

\subsection{Fault detection by computing a theoretical threshold for the residuals}

After learning the NN-KKL observer on the training datasets $\mc{X}_\train\subset\mc{X}$ and $\mc{Z}_\train\subset\mc{Z}$, generate a test dataset $\mc{X}_\test\subset\mc{X}$ by simulating a noise-less ($w\equiv 0, v\equiv 0$) and fault-free ($\phi\equiv 1, \zeta\equiv 0$) system \eqref{system_state}. Then, following the methodology presented in \cite{niazi2022}, we can generate $\mc{Z}_\test \approx \mc{T}(\mc{X}_\test)$, where $\mc{Z}_\test\subset\mc{Z}$. The datasets $\mc{X}_\test$ and $\mc{Z}_\test$ are then used to estimate the approximation errors incurred by $\hat{\mc{T}}_\theta$ and  $\hat{\mc{T}}_\eta^*$. Define
\[
    \hat{\epsilon} \doteq \max_{x\in\mc{X}_\test, z\in\mc{Z}_\test} \left\| z - \hat{\mc{T}}_\theta(x) \right\|
\]
and
\[
    \hat{\epsilon}^* \doteq \max_{x\in\mc{X}_\test, z\in\mc{Z}_\test} \left\| x - \hat{\mc{T}}_\eta^*(z) \right\|.
\]
Then, 
by relying on Proposition~\ref{prop:res_general} for threshold computation, 
we obtain
\begin{align*}
    & \limsup_{t\rightarrow\infty} r_i(t)  \\
    & \qquad \leq \Bar{v} + \ell_{h_i} \bigg[ \xi^*(z) + \frac{\kappa(V)}{c}\|B\|(\ell_h\psi(\Bar{w}) + \sqrt{n_y}\Bar{v}) \bigg].
\end{align*}
Therefore, in this case, we can choose the threshold to be
\begin{equation} \label{eq:taui}
     \boxed{\tau_i = \Bar{v} + \ell_{h_i} \left[ \hat{\epsilon}^* + \frac{\kappa(V)}{c}\|B\|(\ell_h\psi(\Bar{w}) + \sqrt{n_y}\Bar{v}) \right].}
\end{equation}



\textit{Fault detection:} At a steady state when $t$ becomes large, it holds that $r_i(t) \leq \tau_i$ when there are no sensor faults.
When the residual $r_i(t)$ surpasses $\tau_i$, it signifies that the measured output is significantly different from the predicted output. Since $\tau_i$ has been computed using tight inequalities, sensor faults can be detected by measuring the residuals $r_i(t)$ and raising an alert whenever $r_i(t) > \tau_i$ for any $i$ and any $t\geq 0$.

\subsection{Fault isolation by computing an empirical threshold for differentiated residuals} \label{subsec_isolation}

As described in the previous subsection, a sensor fault is detected whenever the residual $r_i(t)$ surpasses the threshold $\tau_i$. However, $r_i(t)$ surpassing $\tau_i$ and having the largest value among other residuals do not mean that sensor~$i$ is faulty\footnote{Assumptions about the triangular structure of the system made in \cite{YAN20071605,tan2003sliding,fonod2014,zhu2022,xu2019conservatism,zhang2017event,su2020fault,vemuri1997neural} would actually enable a simpler treatment because of easy decoupling between the outputs. However, such an assumption holds for a restrictive class of systems.}. This is because a fault in sensor~$j$ distorts the output $y_j(t)$, which is filtered through the observer dynamics and then transformed by $(h_j\circ \hat{\mc{T}}_\eta^*)(\cdot)$. Since the neural network $\hat{\mc{T}}_\eta^*$ is not diagonal, the fault in sensor~$j$ could induce a large residual $r_i(t)$ even when sensor~$i$ is not faulty. This phenomenon can also be observed in the third row of Fig.~\ref{fig:single_sensor} in Section~\ref{simluation_results}. Because of the inter-dependencies due to non-diagonal $\hat{\mc{T}}_\eta^*$, the transients after the occurrence of a fault may persist above the threshold, leading to an inability to exactly isolate the sensor faults from the residuals. Therefore, we propose fault isolation by evaluating differentiated residuals $\Tilde{r}_i(t_k)$.

Based on our empirical observations, a spike is generated in the differentiated residual $\Tilde{r}_i(t_k)$ whenever an abrupt fault ($\phi_i(t)=0$ or $\zeta_i(t)\neq 0$) is introduced in sensor~$i$ at time $t_k$. The explanation for this spike is straightforward because in the differentiated residual, the fault in the output also gets numerically differentiated. Therefore, an abrupt fault results in a spike in the differentiated residual. However, when the fault signal $\zeta_i(t)$ is very smooth and of low magnitude, it is difficult to isolate it. Nevertheless, it is reasonable to assume faults to be irregular, non-smooth signals as compared to a smooth, analytical signal.

Consider the test datasets $\mc{X}_\test$ and $\mc{Z}_\test$, where we denote the state trajectories $x^j(t_k)\doteq x(t_k;x_0^j,0)$ that are initialized at $p$ initial points $x(0)=x_0^j$, for $j=1,\dots,p$, and the corresponding observer estimates as $\hat{x}^j(t_k)\doteq \hat{x}(t_k;\hat{\mc{T}}_\eta^*(z_0^j))$. We simulate the NN-KKL observer and compute the residuals $r_1^j(t_k),\dots,r_{n_y}^j(t_k)$, for $k=0,\dots,T$, where $T\in\bb{N}$ and 
\[
r_i^j(t_k) = |y_i(t_k;x^j) - \hat{y}_i(t_k;\hat{x}^j)|.
\]
Here, $t_T>0$ is chosen large enough to allow the observer to converge to a neighborhood of the original state. Then, the corresponding differentiated residuals $\Tilde{r}_1^j(t_k),\dots,\Tilde{r}_{n_y}^j(t_k)$, for $k=1,\dots,T$, are obtained using \eqref{eq:inc_residual-i}. Let $t_c$, for $c<T$, denote the minimum time at which the observer is in the steady state. Then an empirical threshold for differentiated residuals is computed as follows:
\begin{equation}
    \label{eq:empirical_threshold}
    \boxed{r_\Delta = \max_{\substack{k\in\{c,\dots,T\} \\ j\in\{1,\dots,p\} \\ i\in\{1,\dots,n_y\}}} \Tilde{r}_i^j(t_k).}
\end{equation}

\textit{Fault isolation:} Given that the learned NN-KKL observer estimates the system state accurately, it holds that $\Tilde{r}_i(t)\leq r_\Delta$ at a steady state when $t_k$ becomes sufficiently large. Therefore, whenever the differentiated residual $\Tilde{r}_i(t)$ surpasses the empirical threshold $r_\Delta$ given by \eqref{eq:empirical_threshold}, an alert can be raised that sensor~$i$ is probably faulty.

\section{Simulation Results}\label{simluation_results}

Our approach is demonstrated in this section using numerical simulations. We show that s-FDI of a highly nonlinear system can be achieved using NN-KKL observers, even in the presence of additive process and measurement noises. Both types of faults, sensor degradation represented by the fault signal $\zeta_i(t) \neq 0$ and sensor failure represented by $\phi_i(t) = 0$ in the system output \eqref{system_output}, are demonstrated.

\begin{figure}[!ht]
    \centering
    \includegraphics[width = 0.24\textwidth]{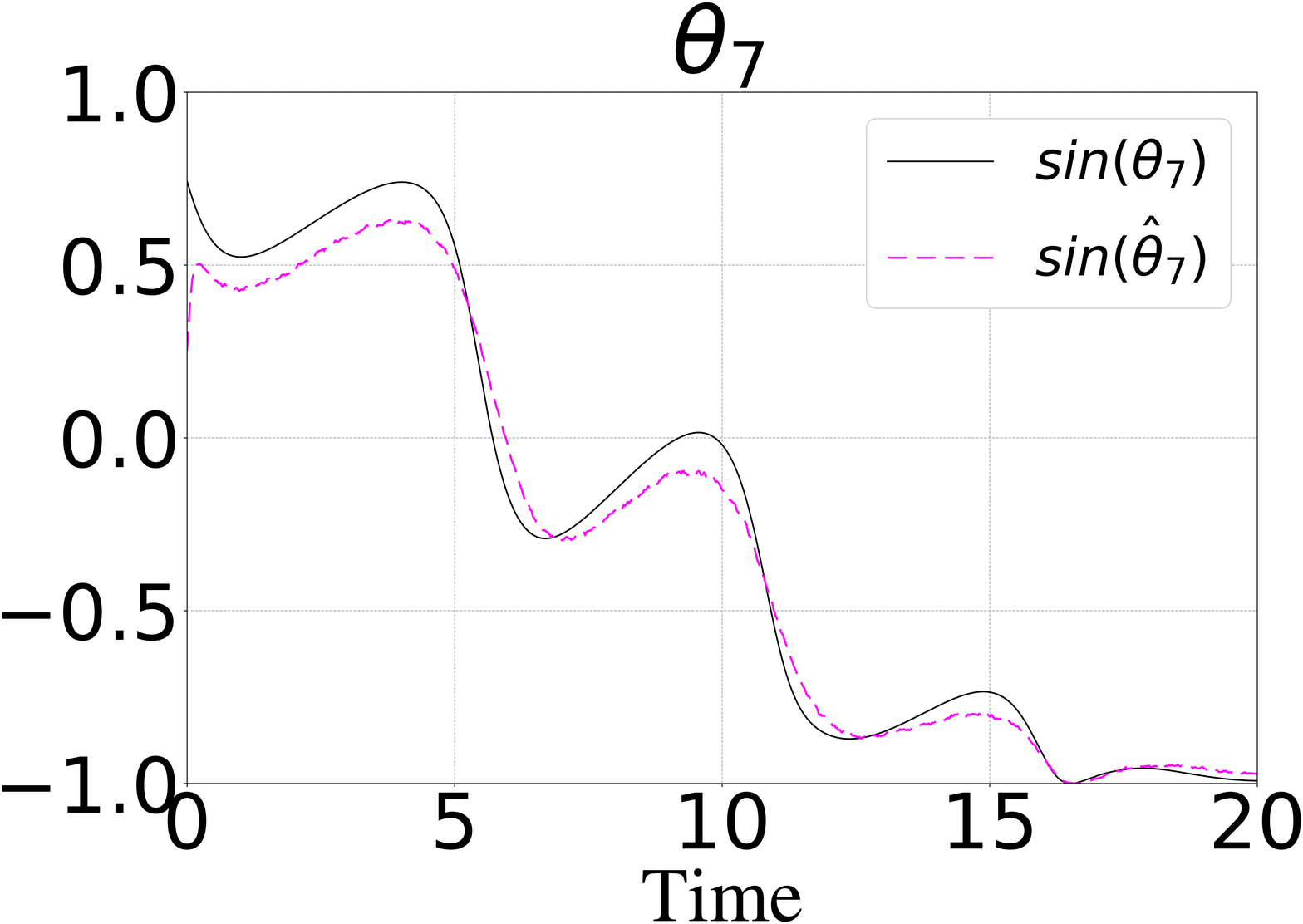}
    \includegraphics[width = 0.24\textwidth]{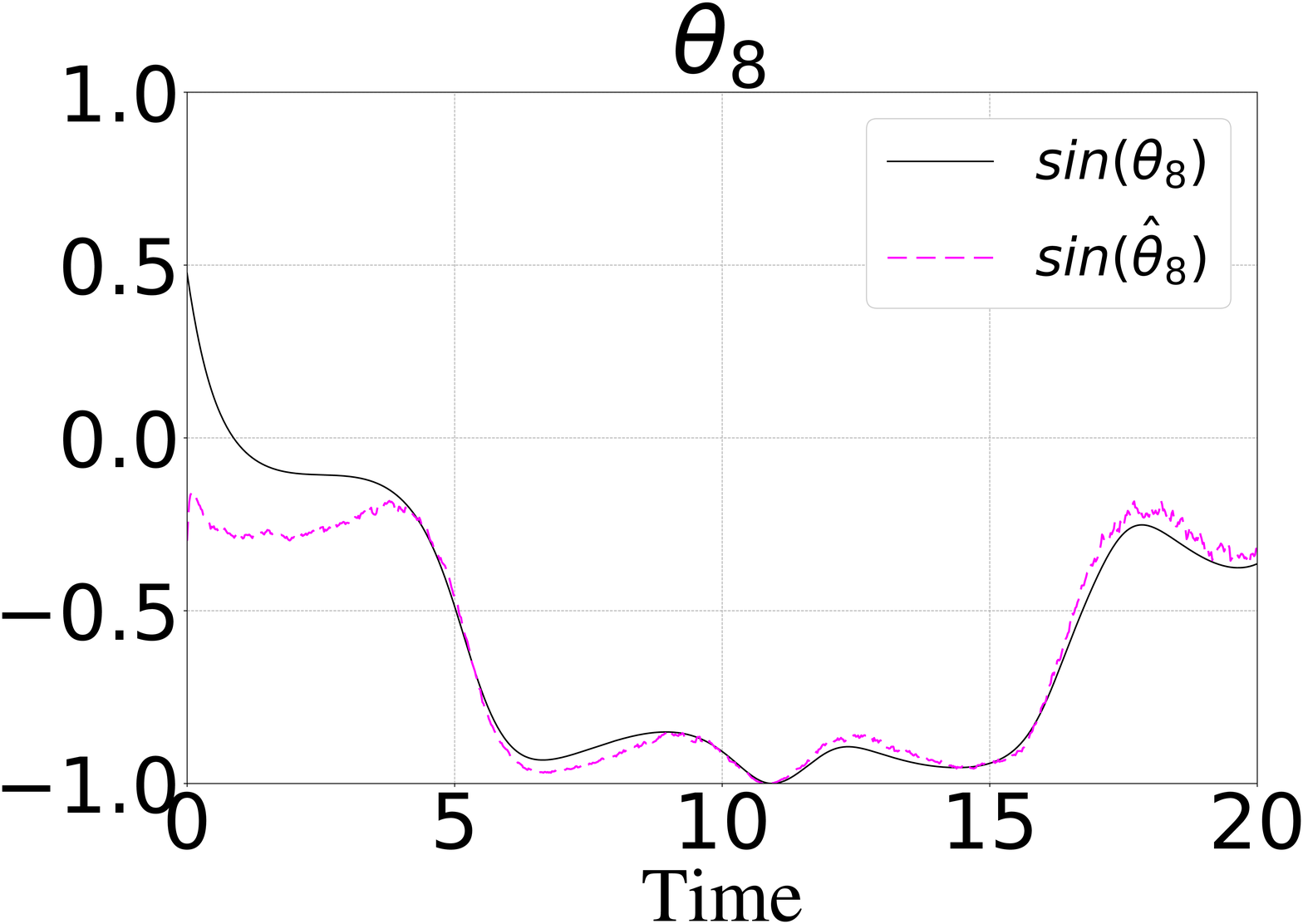}
    
    \caption{Estimated and true trajectories of states $\theta_{7}$ and $\theta_{8}$ under the influence of process and sensor noise.}
    \label{fig:demo}
\end{figure}
\begin{figure*}[!ht]
        \centering
        \begin{subfigure}[t]{0.195\textwidth}
        \includegraphics[width=\textwidth]{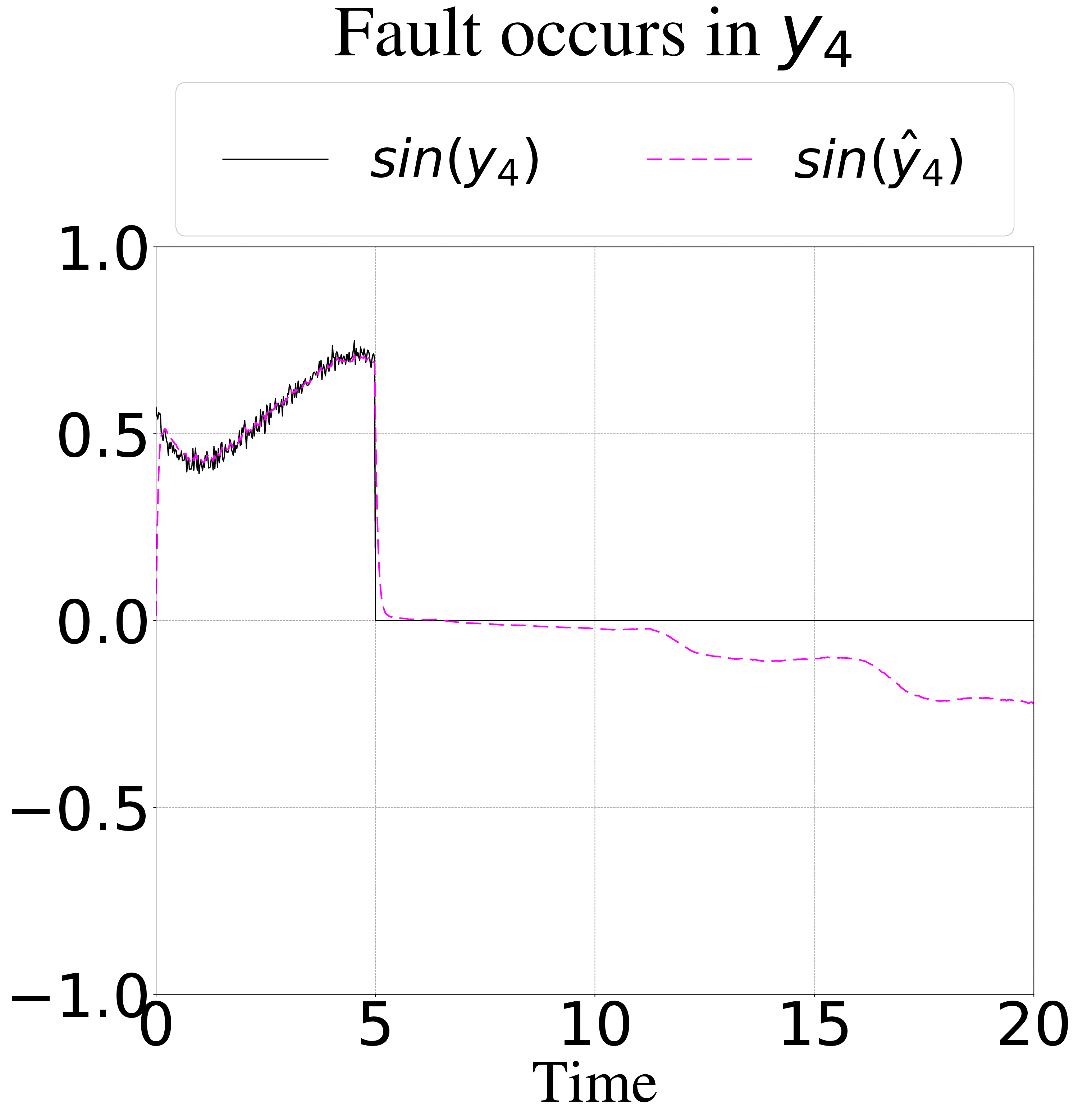}
        \includegraphics[width=\textwidth]{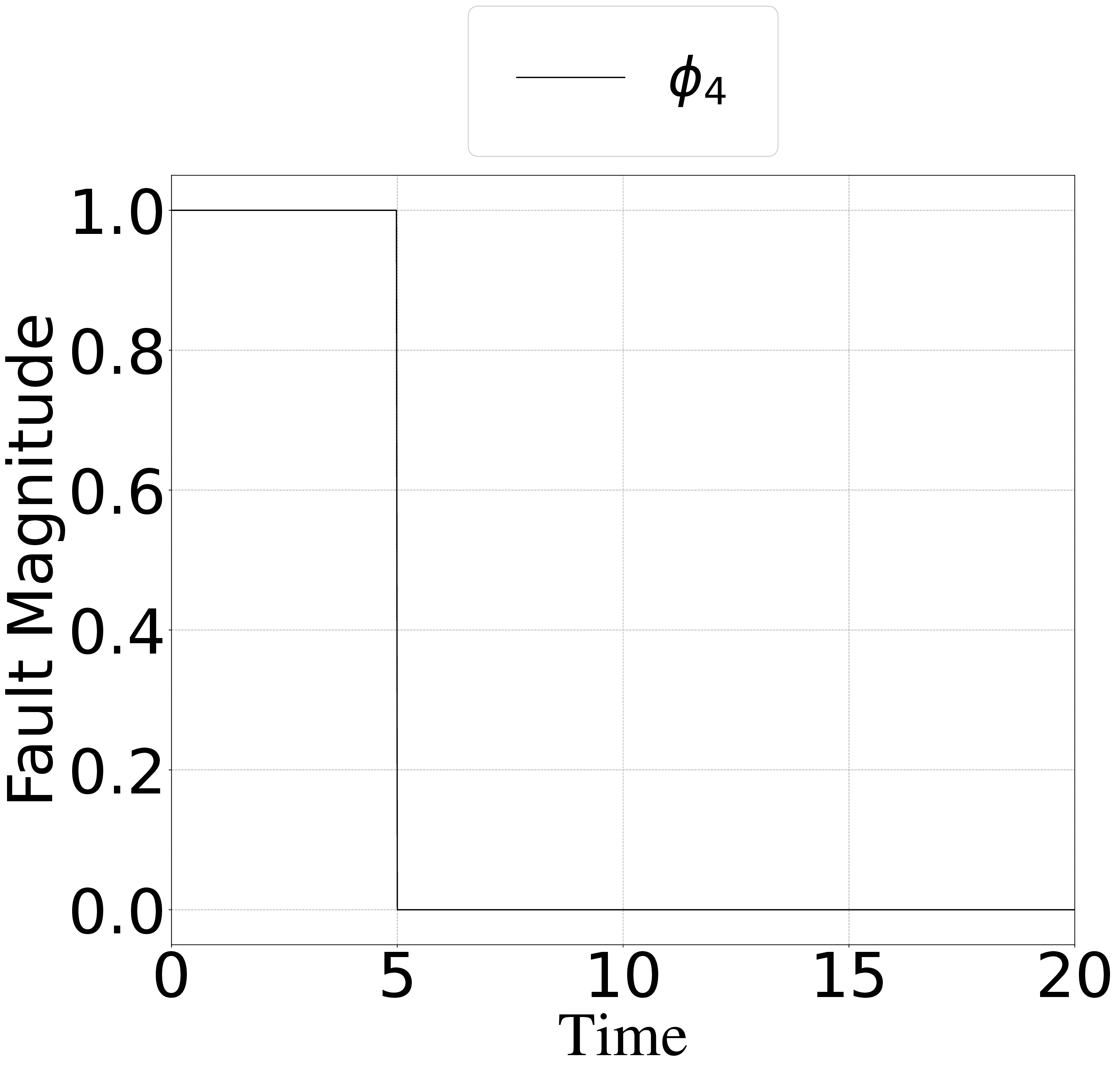}
        \includegraphics[width=\textwidth]{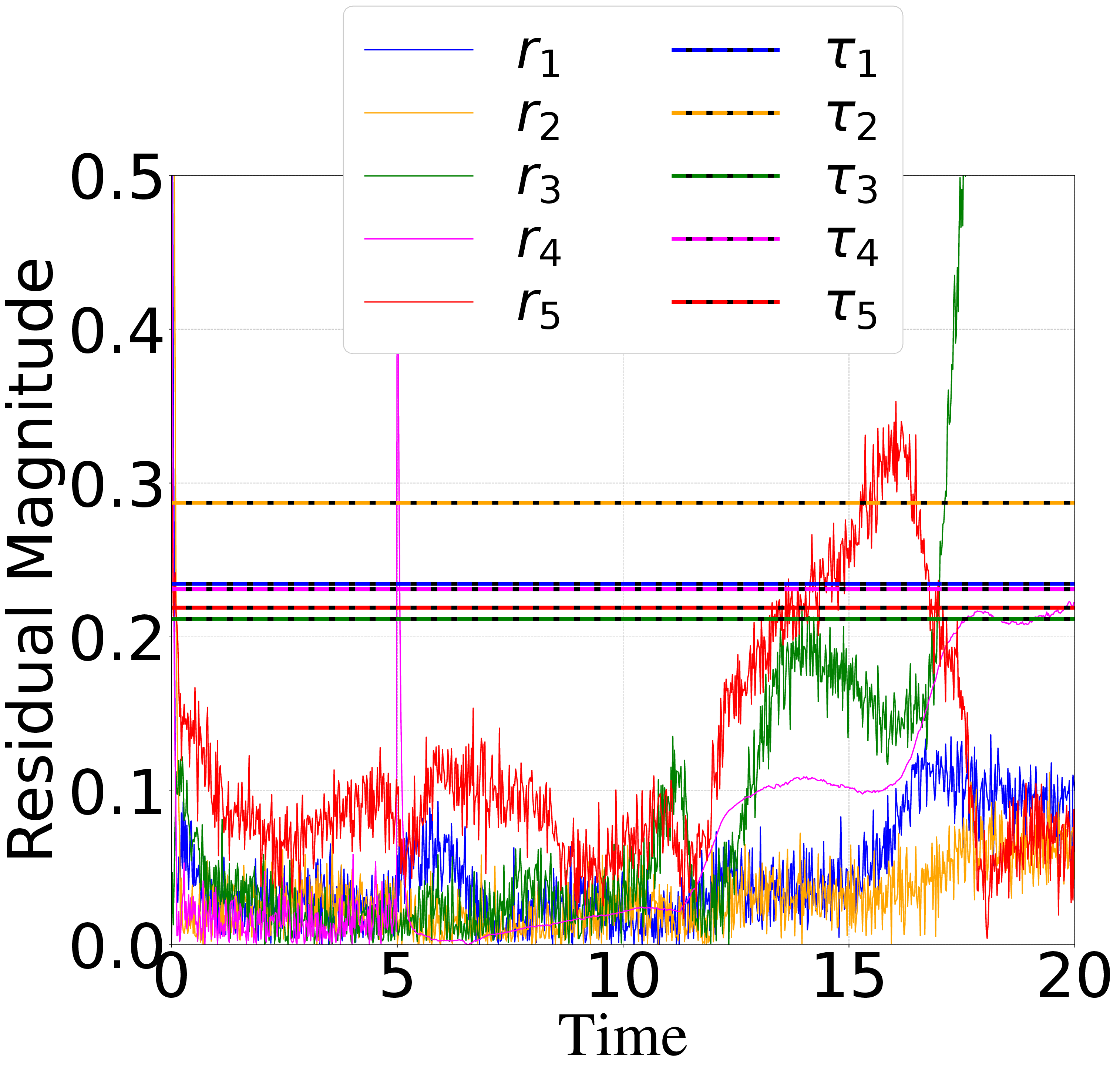}
        \includegraphics[width=\textwidth]{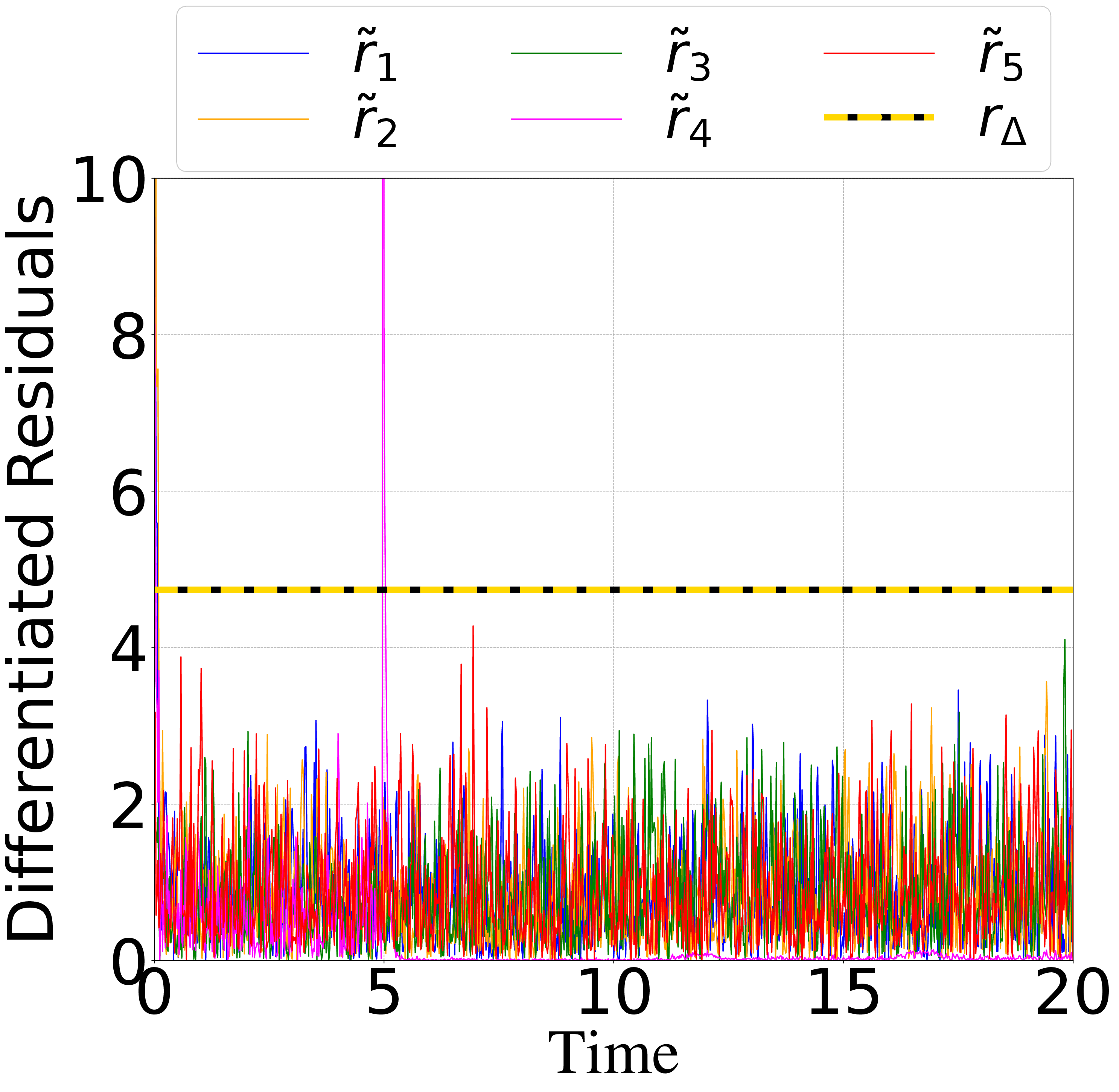}
        \caption{}
        \label{fig:c}
        \end{subfigure}
        \begin{subfigure}[t]{0.195\textwidth}
        \includegraphics[width=\textwidth]{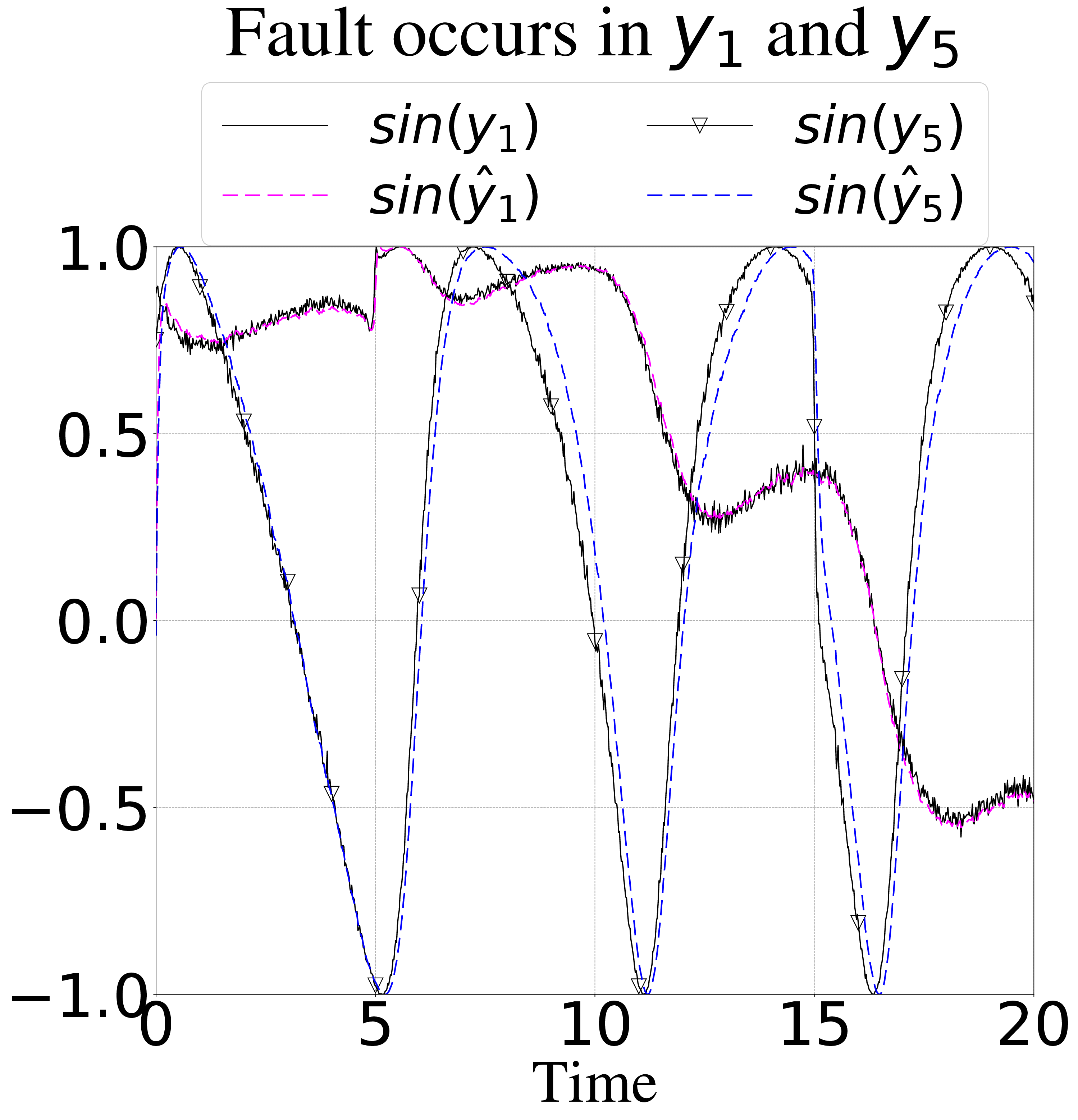}
        \includegraphics[width=\textwidth]{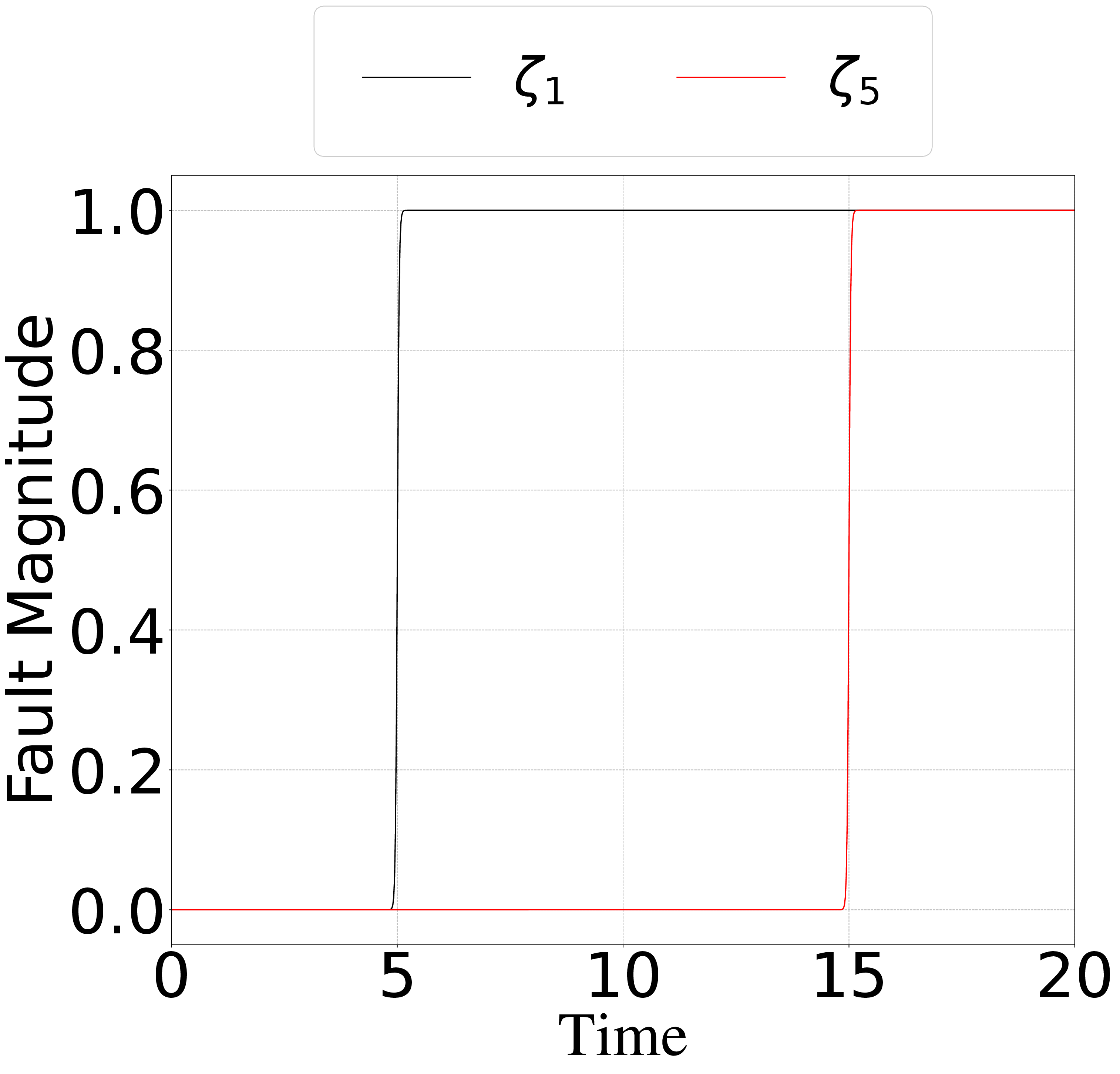}
        \includegraphics[width=\textwidth]{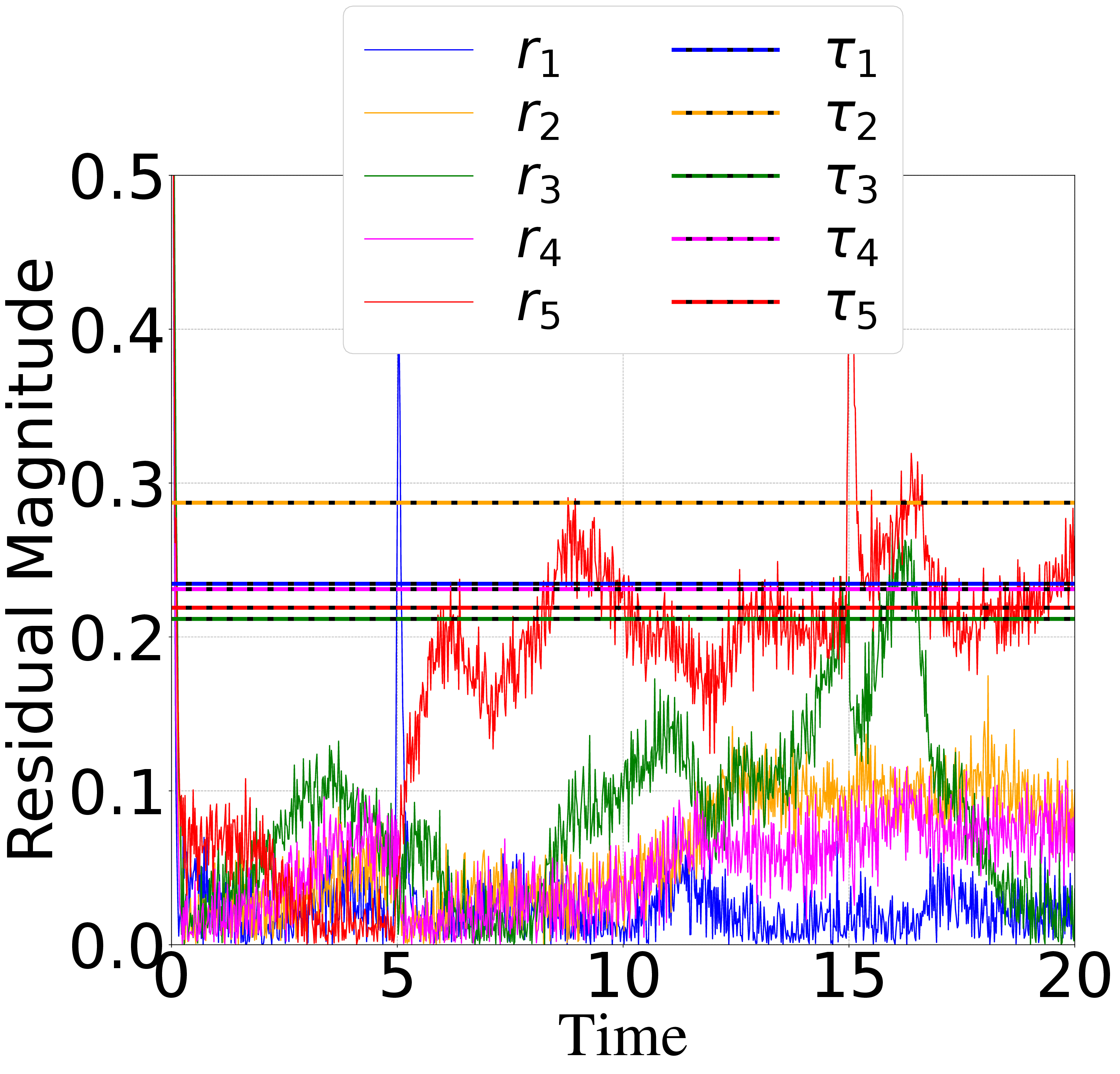}
        \includegraphics[width=\textwidth]{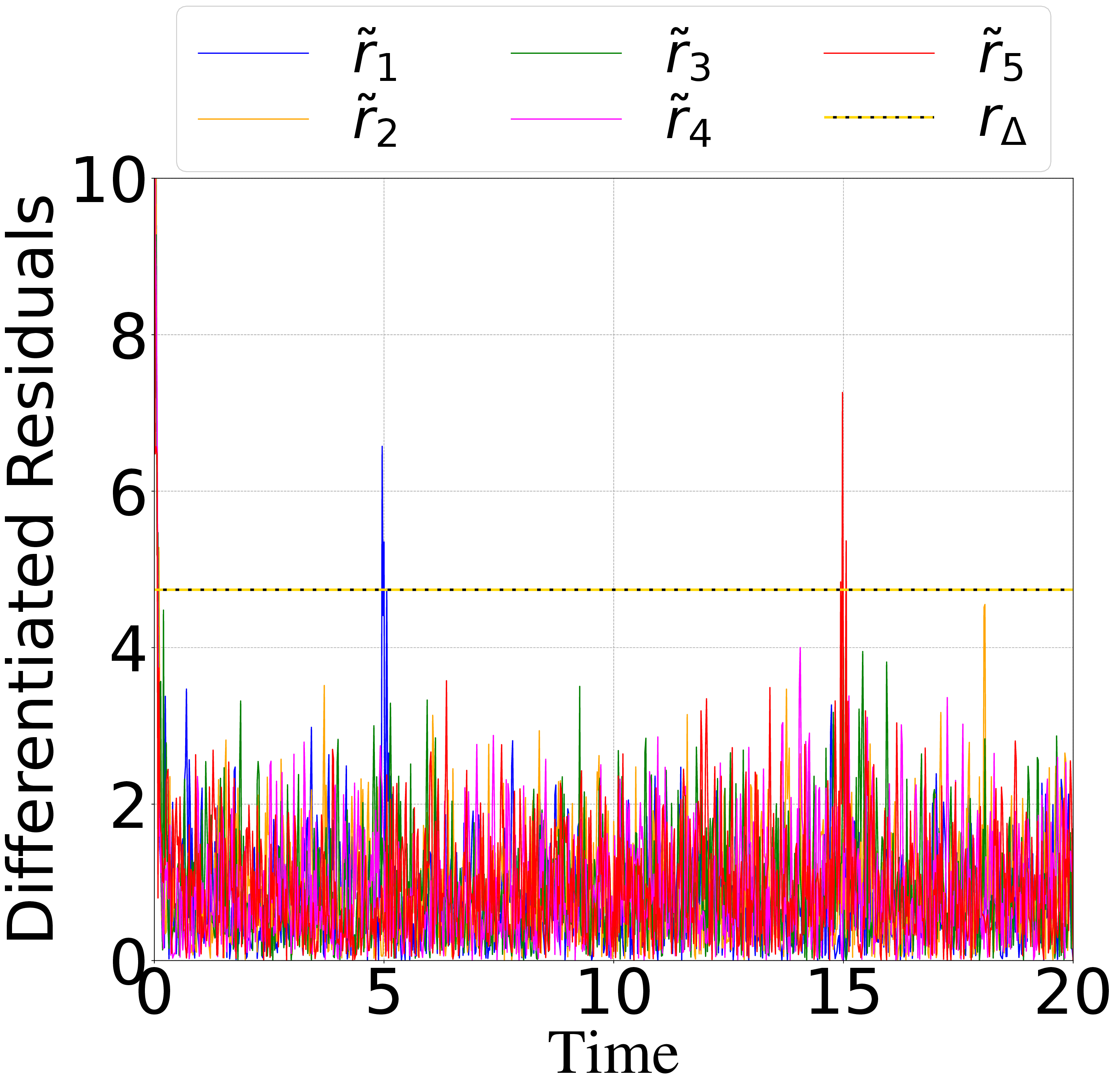}
        \caption{}
        \label{fig:b}
        \end{subfigure}
        \begin{subfigure}[t]{0.195\textwidth}
        \includegraphics[width=\textwidth]{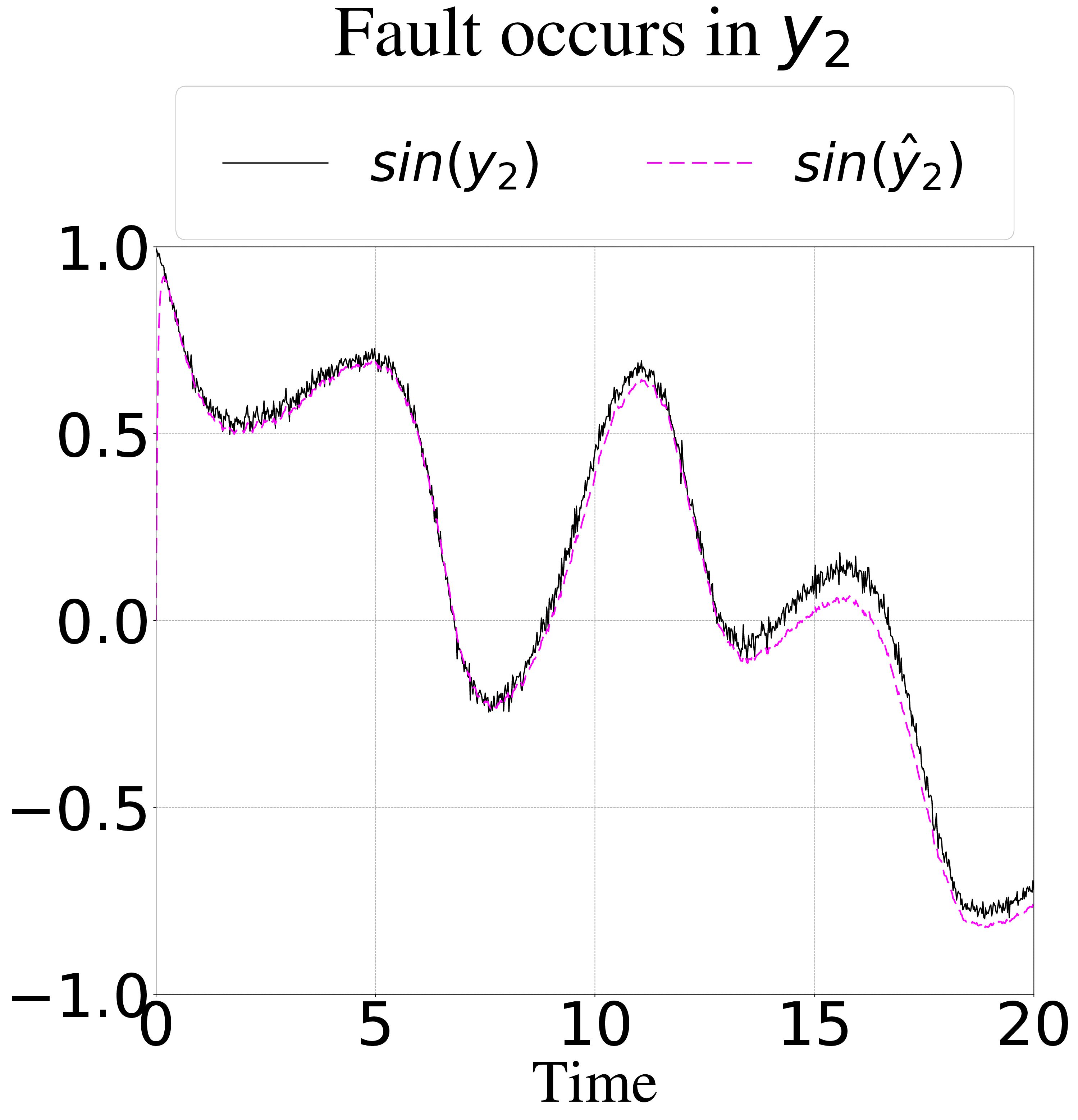}
        \includegraphics[width=\textwidth]{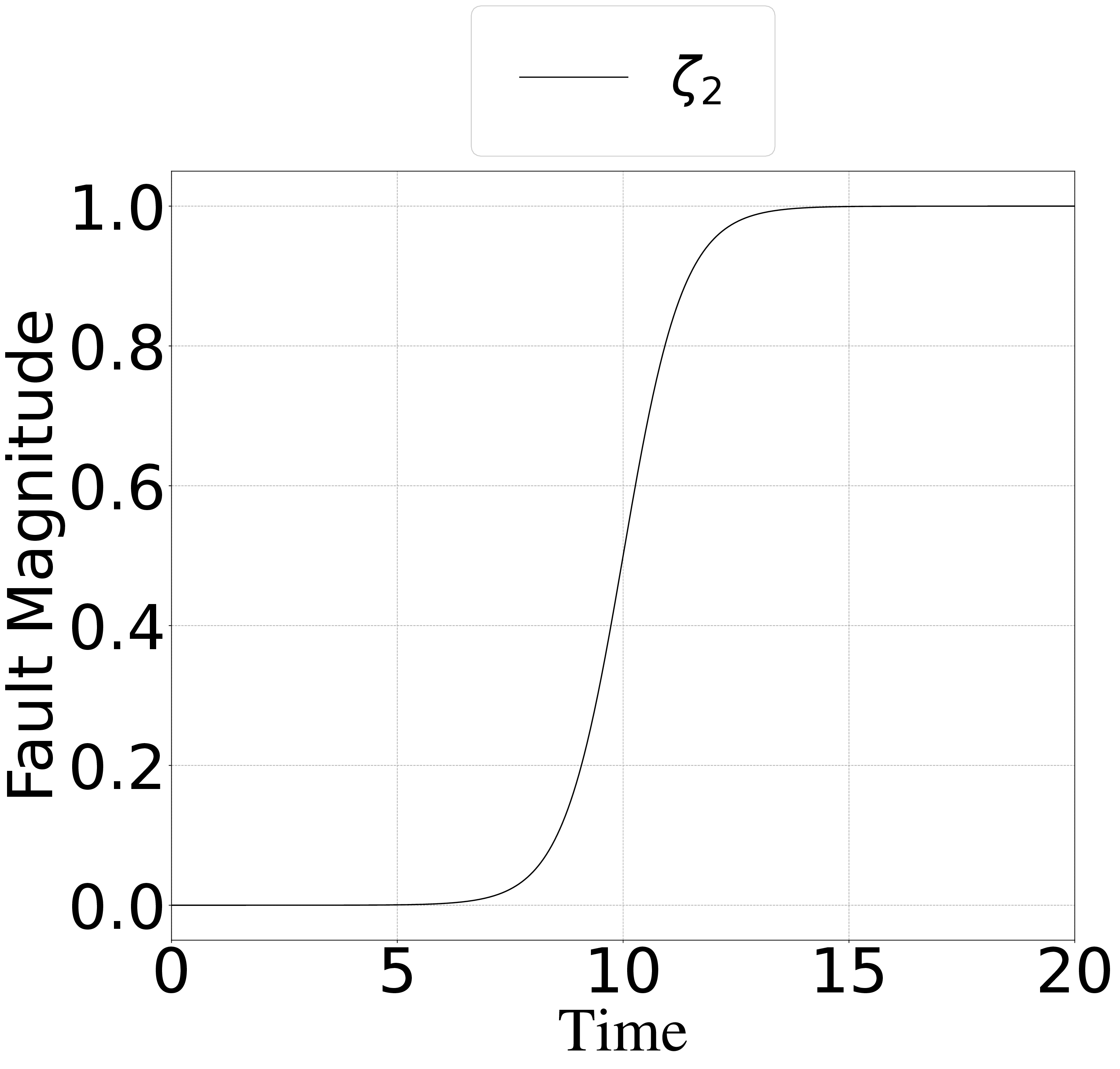}
        \includegraphics[width=\textwidth]{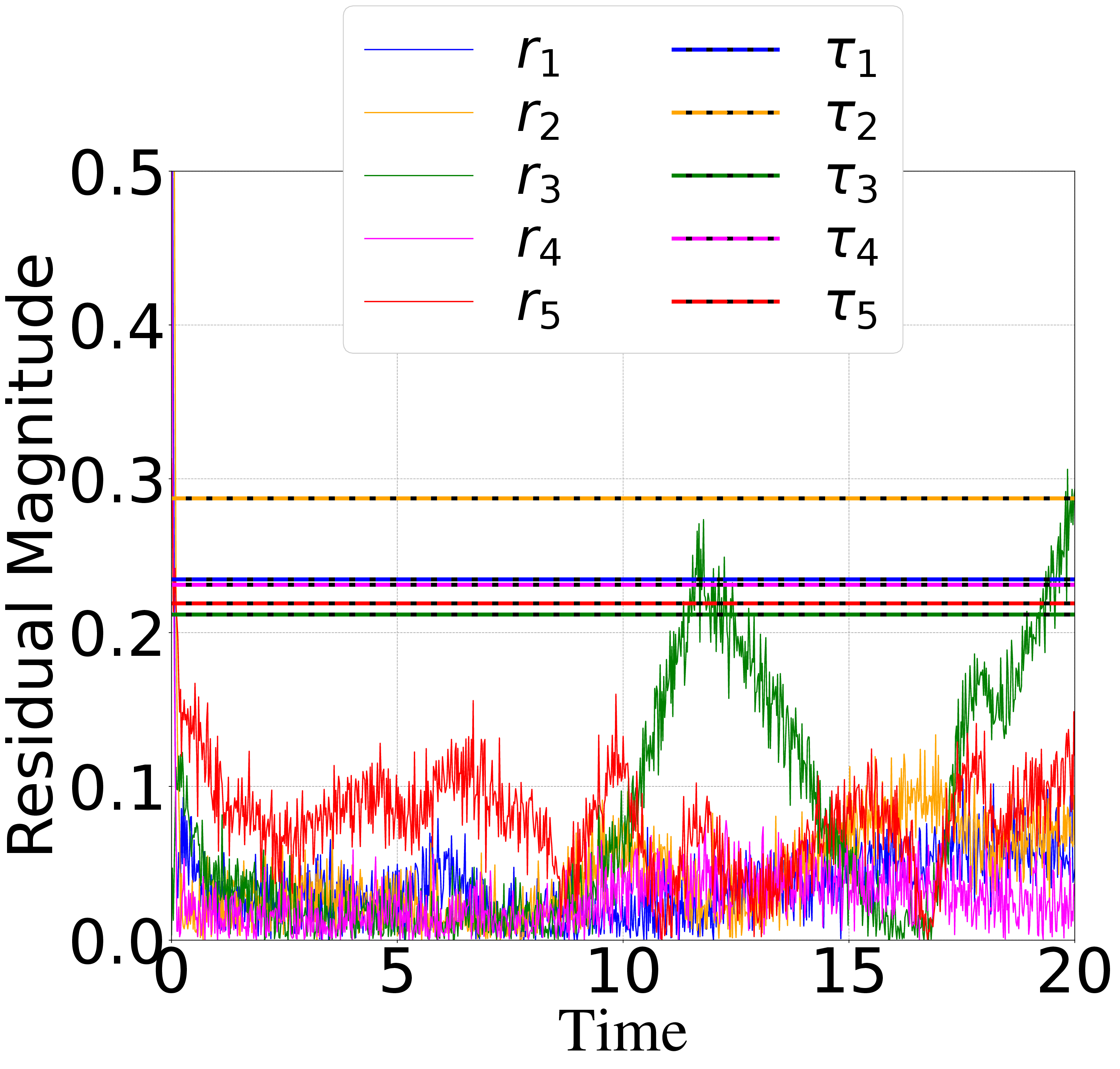}
        \includegraphics[width=\textwidth]{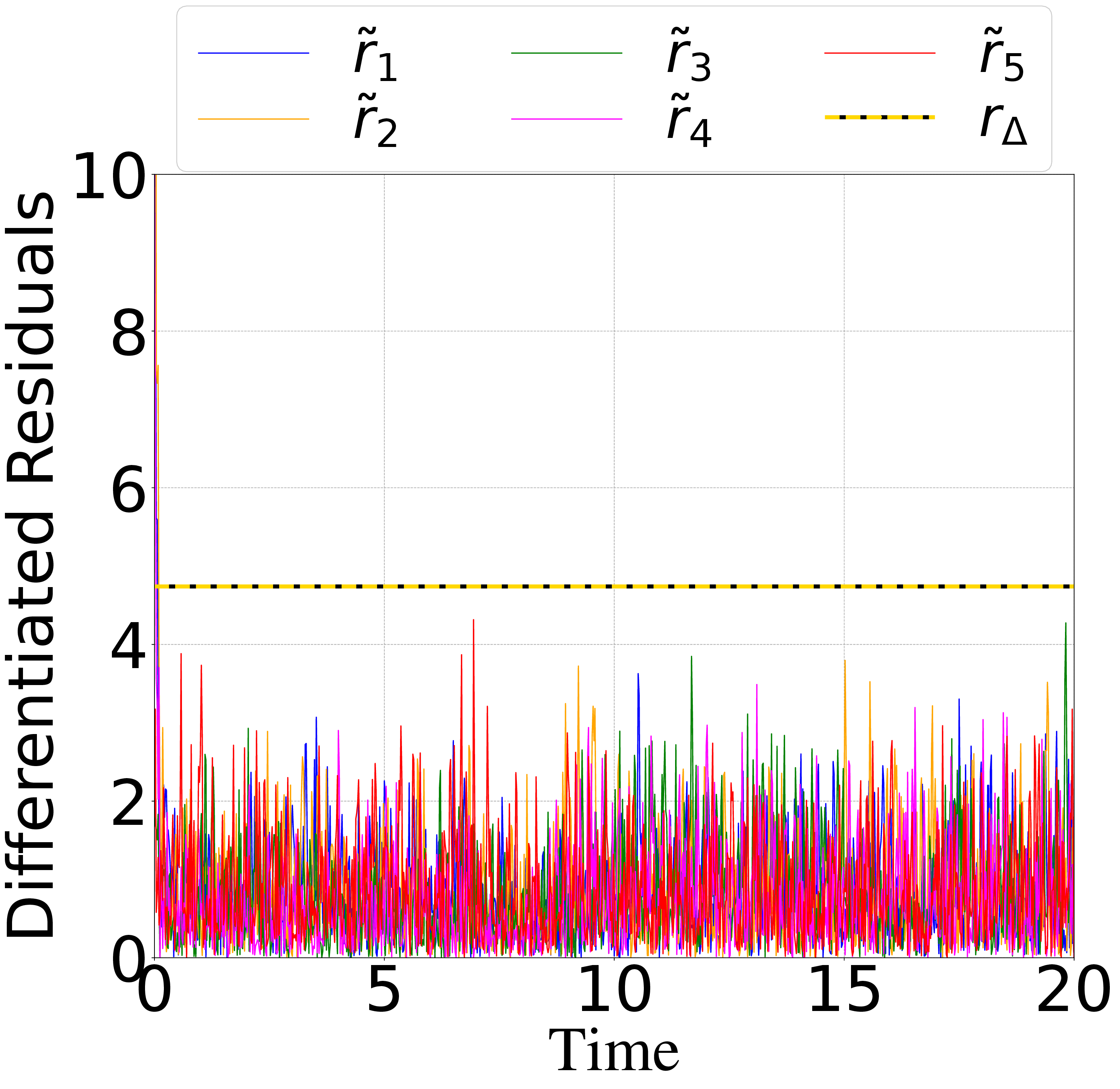}
        \caption{}
        \label{fig:a}
        \end{subfigure}
        \begin{subfigure}[t]{0.195\textwidth}
        \includegraphics[width=\textwidth]{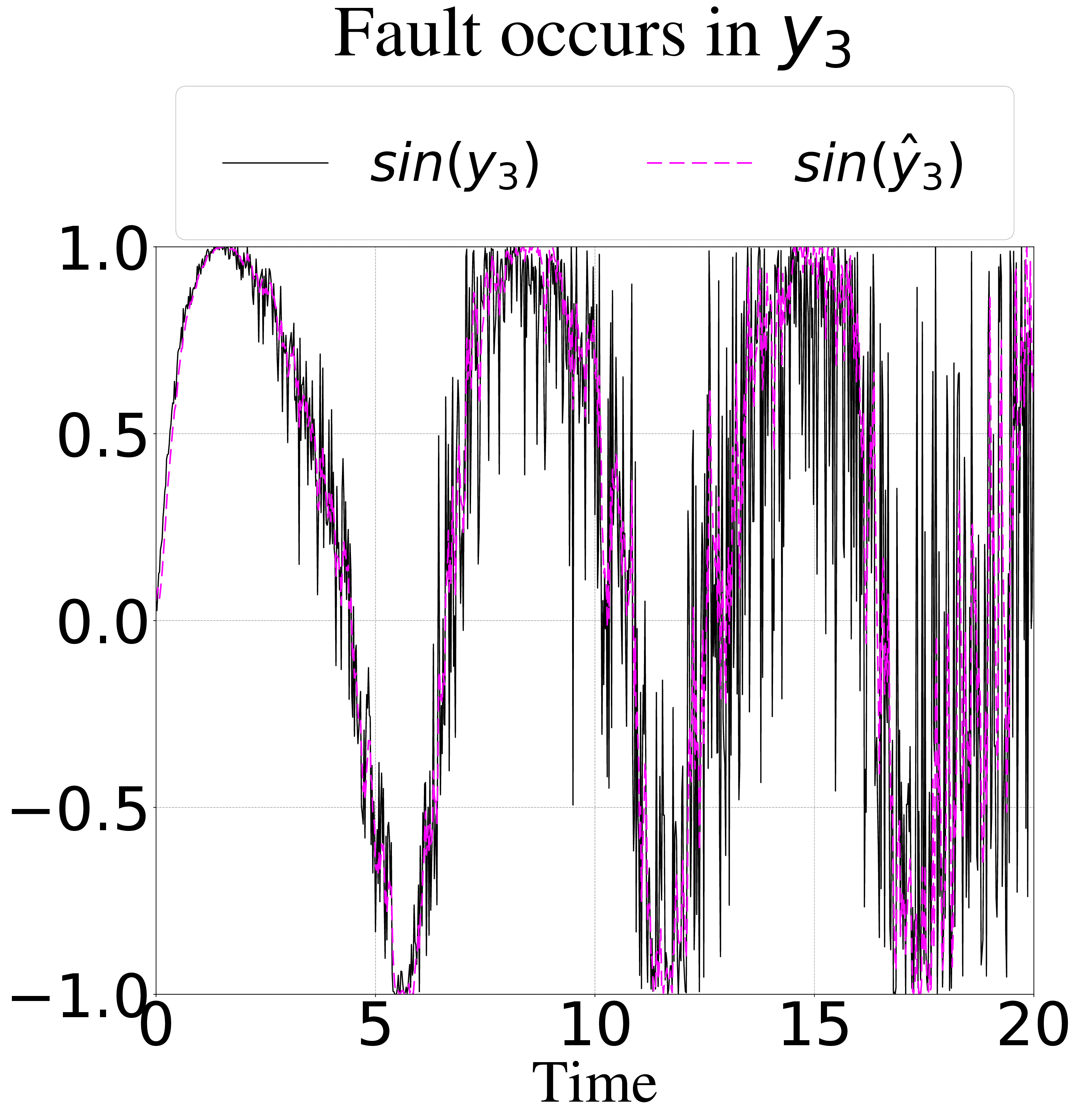}
        \includegraphics[width=\textwidth]{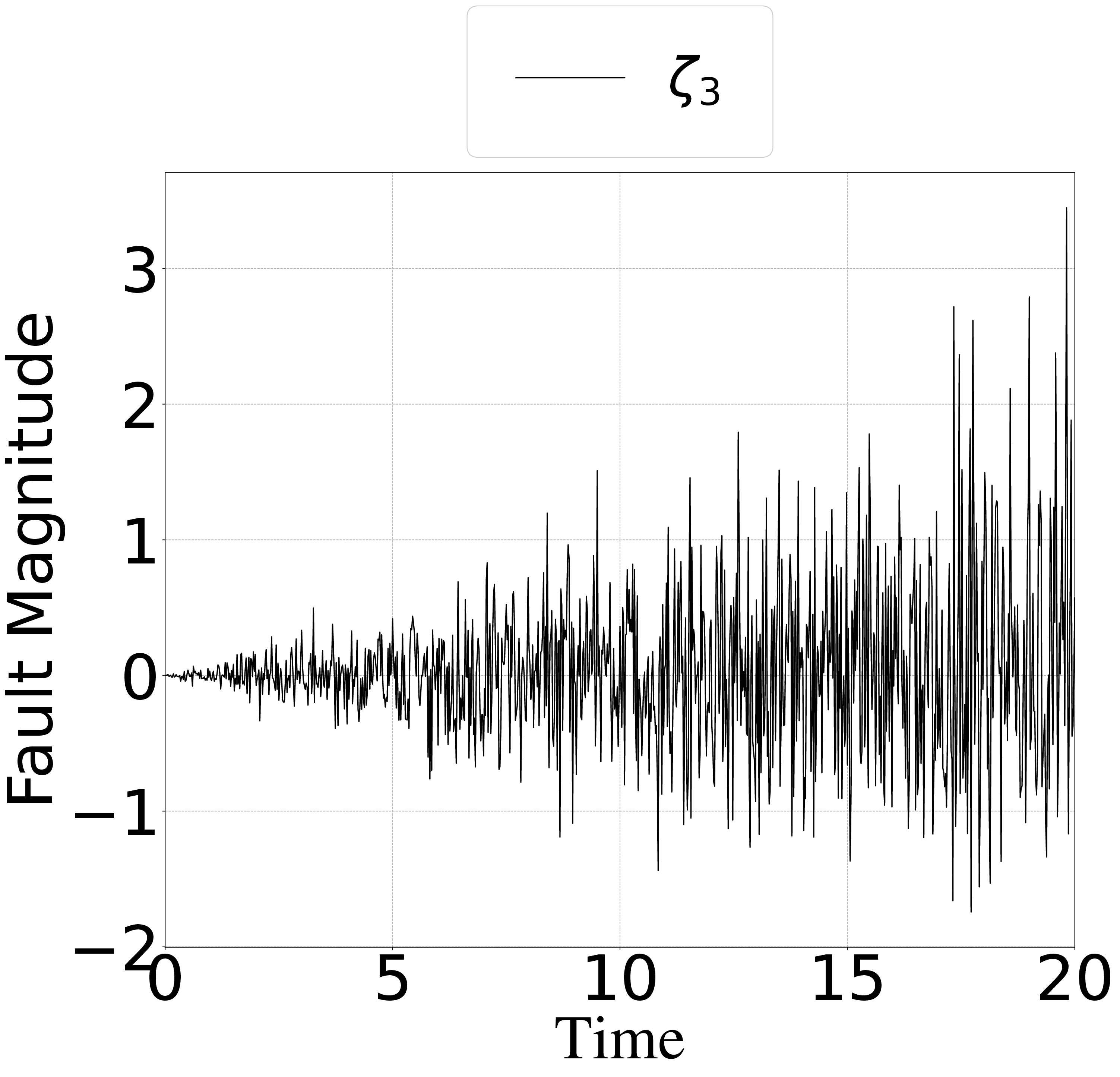}
        \includegraphics[width=\textwidth]{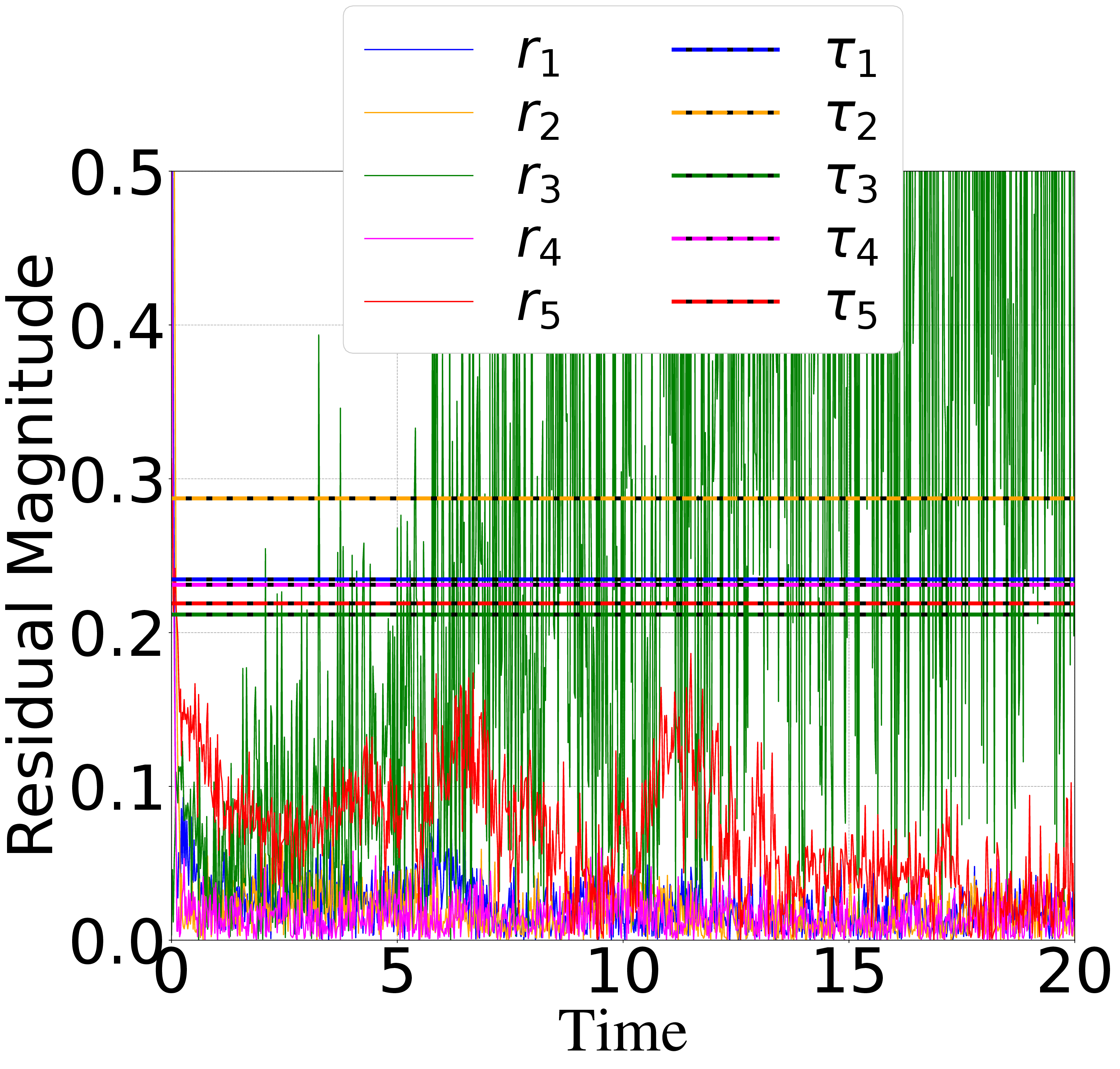}
        \includegraphics[width=\textwidth]{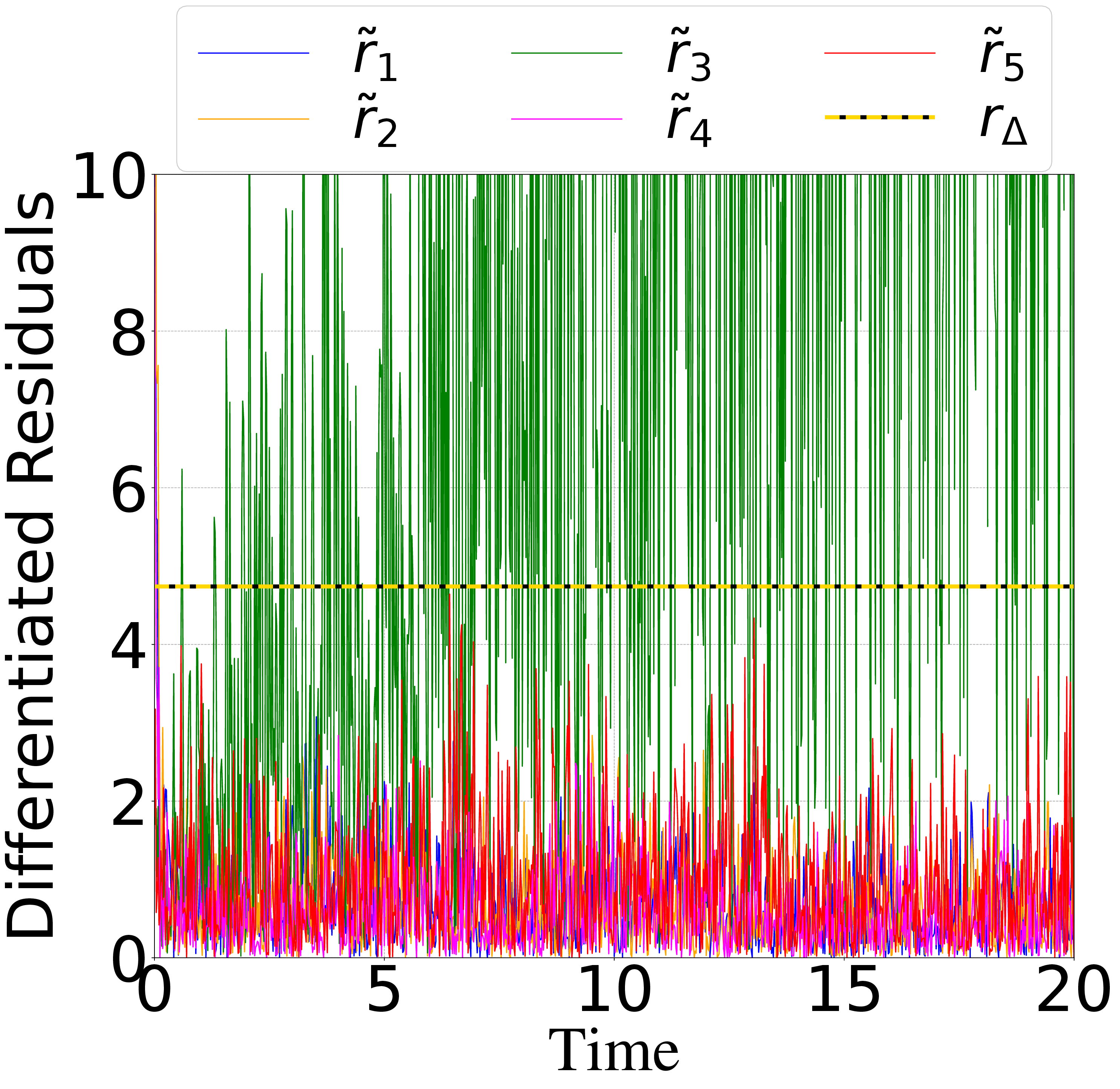}
        \caption{}
        \label{fig:d}
        \end{subfigure}
        \begin{subfigure}[t]{0.195\textwidth}
        \includegraphics[width=\textwidth]{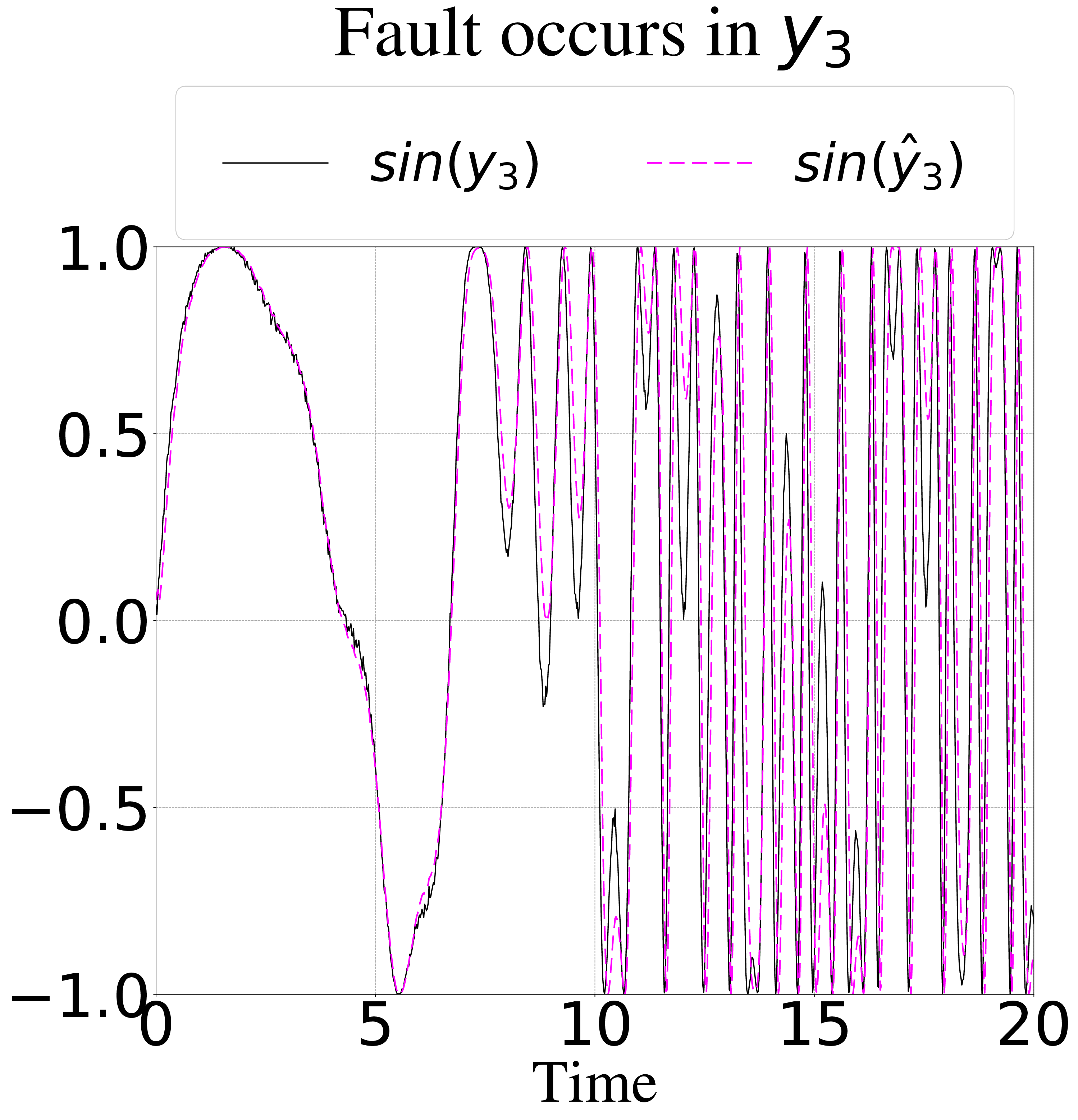}
        \includegraphics[width=\textwidth]{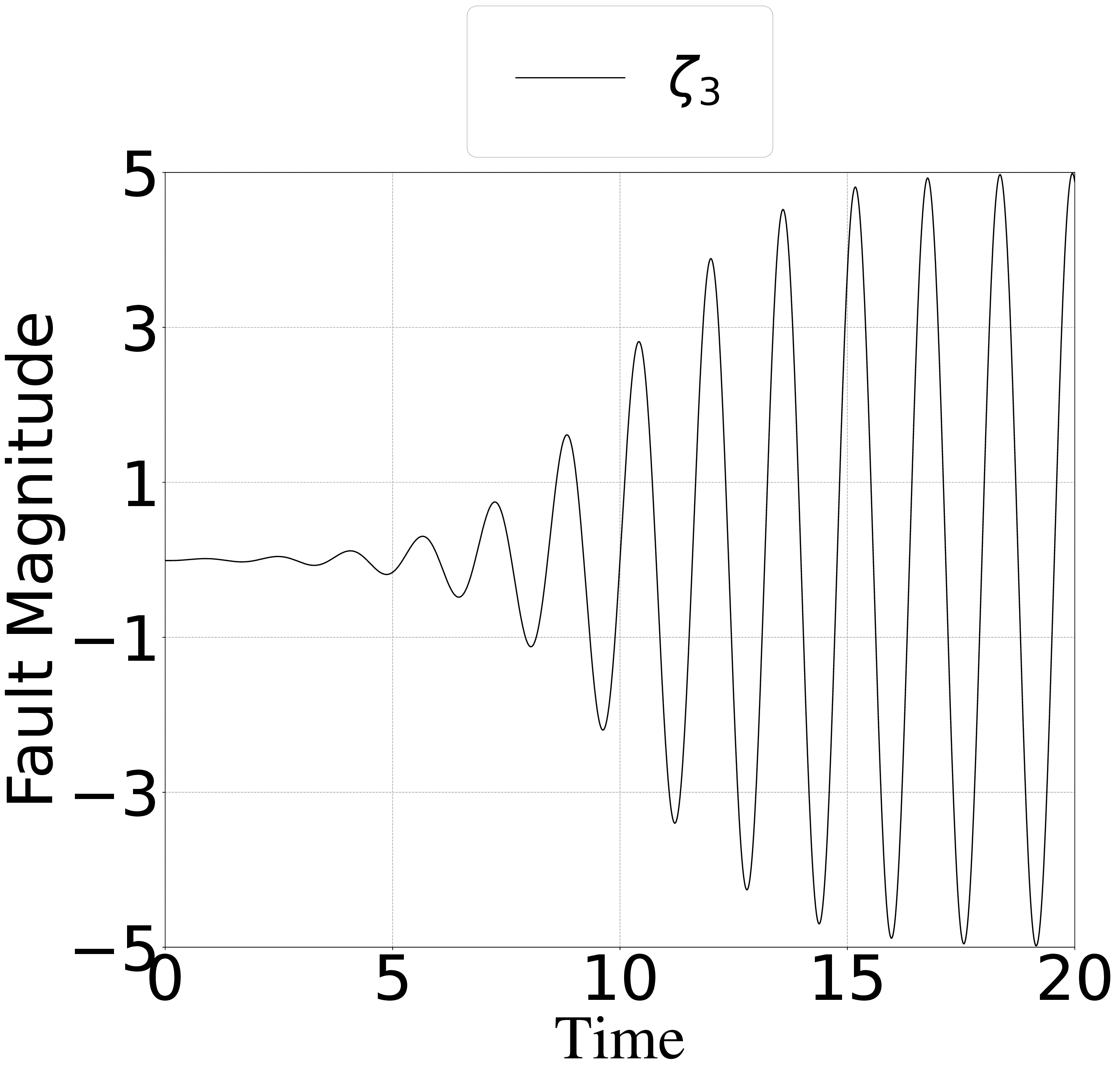}
        \includegraphics[width=\textwidth]{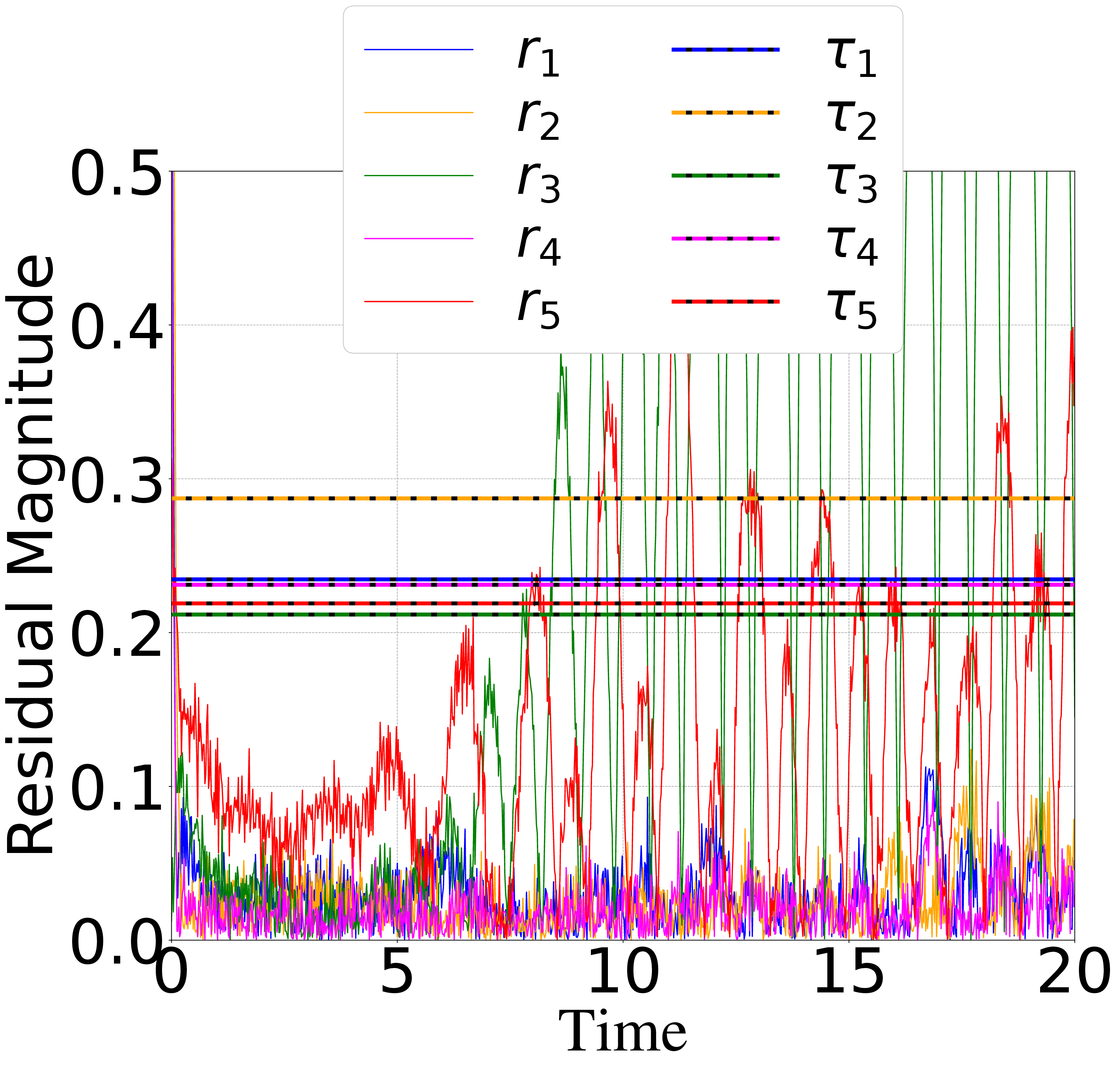}
        \includegraphics[width=\textwidth]{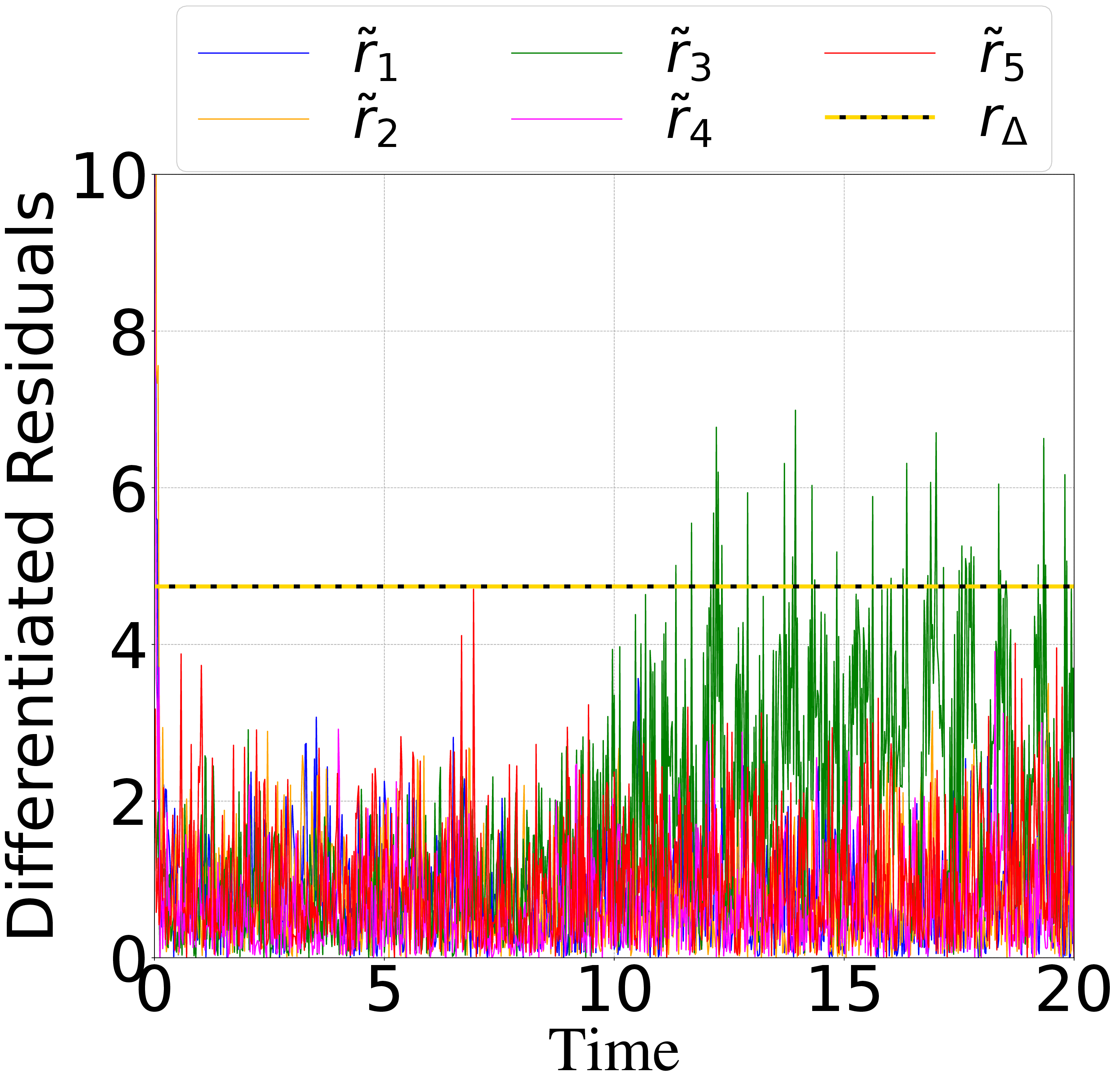}
        \caption{}
        \label{fig:e}
        \end{subfigure}
                
        \caption{(a) Complete failure of sensor 4. (b) Degradation of sensor 1 and 5 at different time points with abrupt faults. (c) Smooth sigmoidal fault in sensor 2. (d) Gradually increasing white noise in sensor 3. (e) Gradually increasing sinusoidal signal in sensor 3.}
        \label{fig:single_sensor}
\end{figure*}

\subsection{Kuramoto Model}

We consider the Kuramoto model for demonstrating our \mbox{s-FDI} method. The Kuramoto model describes the phenomena of synchronization in a multitude of systems, including electric power networks, multi-agent coordination, and distributed software networks \cite{dorfler2014}. The dynamics of a network with $n$ nodes are given as
\begin{equation}\label{kuramoto}
    \dot{\theta_{i}}(t) = \omega_{i} + \sum_{j=1}^{n} a_{ij} \sin(\theta_{i}(t) - \theta_{j}(t))
\end{equation}
where $\theta_{i}(t)\in\mathbb{S}^1$ is the phase angle of node $i = 1,..,n$, $\omega_{i}\in\mathbb{R}$ is the natural frequency of node $i$, and $a_{ij}\geq 0$ denotes the coupling between node $i$ and $j$. In the literature, the state trajectories of \eqref{kuramoto} are often represented graphically as $\sin(\theta_{i})$ in order to better illustrate their synchronization. We follow the same convention in our simulations. 

\subsection{Experimental Setup}

For \eqref{kuramoto}, we consider a network of 10 nodes with randomly generated natural frequencies $\omega_{i}$ and couplings $a_{ij}$. The measurements are chosen as $y= \begin{bmatrix} \theta_{1} & \theta_{2} & \theta_{3} & \theta_{4} & \theta_{5} \end{bmatrix}^\T$. A set of 50 initial conditions is generated using Latin hypercube sampling, with $\mc{X}_{\train} \subset [-2,2]^{10}$. We choose Runge-Kutta-4 as our numerical ODE solver to simulate \eqref{kuramoto} and \eqref{observer_state} over a time interval of $[0,30]$, partitioned into 4000 sample points for each trajectory. The neural network $\hat{\mathcal{T}}^*_{\eta}$ is chosen as a fully connected feed-forward network, consisting of 3 hidden layers of 250 neurons with ReLU activation function. Model training is facilitated by data standardization and learning rate scheduling. Following \cite{bernard2022}, the matrices of \eqref{observer_state} are chosen as
\begin{equation*}
    A = \Lambda \otimes I_{n_{y}}, \qquad B = \Gamma \otimes I_{n_{y}}
\end{equation*}
where $\Lambda \in \mathbb{R}^{(2n_{x}+1) \times (2n_{x}+1)}$ is a diagonal matrix with diagonal elements linearly distributed in $[-15,-21]$, $\Gamma \in \mathbb{R}^{2n_{x}+1}$ is a column vector of ones, and $I_{n_{y}}$ is the identity matrix of size $n_{y} \times n_{y}$. Here, $n_x$ and $n_{y}$ are 10 and 5, respectively, and $n_z = n_y(2n_x+1)=105$.

The estimation capabilities of the observer under fault-free conditions are demonstrated in Fig.~\ref{fig:demo}, which shows the estimated and true trajectories of two (randomly chosen) unmeasured states $\theta_{7}$ and $\theta_{8}$ over a time interval of $[0,20]$, with noise terms $w(t),v(t) \sim \mathcal{N}(0,0.02)$. The figure demonstrates that the estimation error is stable under noise and neural network approximation error.


\subsection{Numerical Results}

We now apply the learned neural network-based KKL observer to perform s-FDI. The theoretical thresholds $\tau_i$ of the residuals $r_i(t)$ computed by \eqref{eq:taui} are shown in the third row of Fig.~\ref{fig:single_sensor}. On the other hand, the empirical threshold $r_\Delta=4.74$ is computed using \eqref{eq:empirical_threshold} according to the method described in Section~\ref{subsec_isolation}. We choose $N = 100$ initial conditions to create $\mathcal{X}_\test \in [-2,2]^{10}$, again using Latin hypercube sampling to compute $r_\Delta$ using \eqref{eq:empirical_threshold}. Fig.~\ref{fig:a}--\ref{fig:e} demonstrate the detection and isolation capabilities of our method under a variety of faults. In the figure, the first, second, third and fourth rows correspond to the measured and estimated state trajectories, the fault signal, the residuals $r_{i} = |y_{i}-\hat{y}_{i}|$ and the finite difference approximation \eqref{eq:inc_residual-i} respectively. 

In Fig.~\ref{fig:c}, we show that our method is capable of detecting sensor shutdowns due to complete failure, which we demonstrate by modeling the fault in sensor~${4}$ with $\phi_{4}(t) = 0$. Fig.~\ref{fig:b} illustrates the situation when more than one fault is present. Sensors~${1}$ and ${5}$ are disturbed by $\zeta_{1}(t)=1$, at $t\geq 5$, and $\zeta_{5}(t) = 1$, at $t \geq 15$. Each fault is distinctly detectable at the moment of occurrence.

In Fig.~\ref{fig:a}, sensor~${2}$ is disturbed by a smooth sigmoidal fault term $\zeta_{2}(t)$ introduced at $t = 5$. Because of the smooth sigmoidal signal as fault, the observer does not detect any anamoly and continues to follow the measured state trajectories, thus generating a small residual. Although introducing an abrupt fault as in Fig.~\ref{fig:c} and \ref{fig:b} induces a transient in the observer, causing a large residual to be generated, introducing a smooth fault in Fig.~\ref{fig:a} resulted in only detection but not isolation. Due to the stability of the observer, it attempts to track the faulty trajectories after the occurrence of the fault, leading to a decrease in the residual magnitudes subsequently.

Fig.~\ref{fig:d} and \ref{fig:e} illustrates the case when the fault signal on sensor~$3$ is introduced gradually. In Fig.~\ref{fig:d}, we simulate a fault in sensor $y_{3}$ which is a gradually increasing white noise. In Fig.~\ref{fig:e}, the fault is a sinusoidal signal that gradually increases its magnitude from 0 to 5. Here, due to a gradual increase in the fault signals' magnitudes, it can be observed that both detection and isolation are successfully performed but with a small delay.


\section{Conclusion and Future Work} \label{conclusion}

We have introduced a novel sensor fault detection and isolation method tailored for a general class of nonlinear systems. For s-FDI, we leverage a neural network-based Kazantzis-Kravaris/Luenberger (KKL) observer for residual generation. We derived theoretical upper bounds for the residuals that are obtained by comparing measured outputs with those predicted by the observer. These upper bounds serve as the foundation for analytically computing thresholds. When a residual crosses its corresponding threshold, it indicates the detection of sensor faults. However, when it comes to the critical task of fault isolation, specifically pinpointing the malfunctioning sensor, we observed the importance of examining numerically differentiated residuals rather than the usual residuals. To this end, we introduced an empirical methodology, involving experimentation with the learned KKL observer, to calculate the threshold for the differentiated residuals.

We demonstrated the efficacy of our approach on a network of Kuramoto oscillators by evaluating various fault scenarios, including sensor degradation and complete failure. The comprehensive set of simulations provides compelling evidence of the method's robustness and effectiveness in both fault detection and isolation within the context of autonomous nonlinear systems.

The theory of KKL observers extends to non-autonomous systems, and adapting our method to those systems remains an open research topic. It is also of interest to study the performance of the method in real applications, especially in systems where conventional solutions are known to fail.


\appendix

\subsection{Proof of Proposition~\ref{prop:res_contraction}}
\label{appendix:proof_prop1}
    In a fault-less case, $\phi(t)=1$ and $\zeta(t)=0$. Thus, from \eqref{eq:residual-i},
    \begin{align}
        |r_i| & = |h_i(x)-h_i(\hat{x}) + v_i| \notag \\
        & \leq \ell_{h_i} \|x-\hat{x}\| + |v_i|. \label{eq:prop_proof1}
    \end{align}
    We have
    \begin{align}
        \|x-\hat{x}\| & = \|\mc{T}^*(z) - \hat{\mc{T}}_\eta^*(\hat{z})\| \notag \\
        & = \|\mc{T}^*(z) - \hat{\mc{T}}_\eta^*(z) + \hat{\mc{T}}_\eta^*(z) - \hat{\mc{T}}_\eta^*(\hat{z})\| \notag \\ 
        & \leq \|\mc{T}^*(z) - \hat{\mc{T}}_\eta^*(z)\| + \|\hat{\mc{T}}_\eta^*(z) - \hat{\mc{T}}_\eta^*(\hat{z})\| \notag \\ 
        & \leq \xi^*(z) + \ell_\eta \|z-\hat{z}\|. \label{eq:prop_proof2}
    \end{align}
    Similarly, we have
    \begin{align}
        \|z-\hat{z}\| & = \|\mc{T}(x) - \hat{\mc{T}}_\theta(\hat{x}) \pm \hat{\mc{T}}_\theta(\hat{x}) \| \notag \\ 
        & \leq \|\mc{T}(x)-\hat{\mc{T}}_\theta(x)\| + \|\hat{\mc{T}}_\theta(x) - \hat{\mc{T}}_\theta(\hat{x})\| \notag \\
        & \leq \xi(x) + \ell_\theta \|x-\hat{x}\|. \label{eq:prop_proof3}
    \end{align}
    By substituting \eqref{eq:prop_proof3} into \eqref{eq:prop_proof2} and rearranging, we obtain
    \[
        \|x-\hat{x}\| \leq \frac{\xi^*(z) + \ell_\eta\xi(x)}{1-\ell_\theta\ell_\eta}.
    \]
    Substituting this into \eqref{eq:prop_proof1} and using Assumption~\ref{assume-1}(i) proves the result.

\subsection{Proof of Lemma~\ref{lemma:zeta_bound}}
\label{appendix:proof_lemma2}

Define $\sigma(t)\doteq h(\Bar{x}(t))-h(x(t))-v(t)$. Then, from the dynamics $\dot{\Tilde{z}}(t)$, we have that the solution $\Tilde{z}(t;\Tilde{z}_0,\sigma)$ satisfies
\be
    \label{eq:ineq_ztilde}
    \|\Tilde{z}(t)\| \leq \|\exp(At)\Tilde{z}_0\| + \int_0^t \|\exp(A\tau)B\sigma(t-\tau)\| \text{d}\tau.
\ee
\begin{lemma} \label{lemma:exp_ineq}
    Suppose the matrix $A$ is Hurwitz and diagonalizable with eigenvalue decomposition $A=V\Lambda V^{-1}$. Then,
    \[
        \|\exp(At)\| \leq \kappa(V) e^{-ct}
    \]
    and
    \[
        \int_0^t \|\exp(A\tau) B \| \text{d}\tau \leq \frac{\kappa(V)}{c} \|B\| (1-e^{-ct})
    \]
    where $\kappa(V)=\|V\|\|V^{-1}\|$ is the condition number of $V$ and
    \be \label{eq:min_evalue}
        c = \min_{\lambda\in\eig(A)} |\text{Re}(\lambda)|.
    \ee
\end{lemma}

The proof of this lemma is straightforward. A similar result when $A$ is not diagonalizable can be obtained via the Jordan form of $A$; see \cite[Appendix C.5]{sontag2013}. However, since diagonalizable matrices are dense in the space of square matrices, the diagonalizability assumption is very mild.

The second term on the right-hand side of \eqref{eq:ineq_ztilde} satisfies
\begin{align}
& \int_0^t \|\exp(A\tau)B\sigma(t-\tau)\| \text{d}\tau \hspace{2cm} \nonumber \\ 
& \hspace{2cm} \leq \int_0^t \|\exp(A\tau)B\| \|\sigma(t-\tau)\| \text{d}\tau \nonumber \\
& \hspace{2cm} \leq \int_0^t \|\exp(A\tau)B\| \text{d}\tau \, \max_{\tau\geq 0} \|\sigma(\tau)\|
\label{eq:ineq_2}
\end{align}
From the definition of $\sigma(t)$, we have
\[\arraycolsep=2pt
\ba{ccl}
\|\sigma(t)\| & \leq & \|h(\Bar{x}(t))-h(x(t))\| + \|v(t)\| \\
& \leq & \ell_h \|\Bar{x}(t)-x(t)\| + \|v(t)\| \\
& \leq & \ell_h \psi(\|w_{[0,t]}\|_{L^\infty}) + \sqrt{n_y} \|v(t)\|_\infty.
\ea\]
The first step is due to the triangle inequality, the second step is due to the smoothness of $h(\cdot)$, and the last step is due to Assumption~\ref{assume-1}(ii) and properties of vector norms. Therefore, 
\be
    \label{eq:ineq_max}
    \max_{\tau\geq 0} \|\sigma(\tau)\| \leq \ell_h \psi(\Bar{w}) + \sqrt{n_y} \Bar{v}.
\ee
The result then follows by substituting \eqref{eq:ineq_2} and \eqref{eq:ineq_max} into \eqref{eq:ineq_ztilde}, and then using the Lemma~\ref{lemma:exp_ineq}.

\bibliographystyle{IEEEtran}
\bibliography{bib}

\begin{thebibliography}{10}
\providecommand{\url}[1]{#1}
\csname url@samestyle\endcsname
\providecommand{\newblock}{\relax}
\providecommand{\bibinfo}[2]{#2}
\providecommand{\BIBentrySTDinterwordspacing}{\spaceskip=0pt\relax}
\providecommand{\BIBentryALTinterwordstretchfactor}{4}
\providecommand{\BIBentryALTinterwordspacing}{\spaceskip=\fontdimen2\font plus
\BIBentryALTinterwordstretchfactor\fontdimen3\font minus
  \fontdimen4\font\relax}
\providecommand{\BIBforeignlanguage}[2]{{%
\expandafter\ifx\csname l@#1\endcsname\relax
\typeout{** WARNING: IEEEtran.bst: No hyphenation pattern has been}%
\typeout{** loaded for the language `#1'. Using the pattern for}%
\typeout{** the default language instead.}%
\else
\language=\csname l@#1\endcsname
\fi
#2}}
\providecommand{\BIBdecl}{\relax}
\BIBdecl

\bibitem{VENKATASUBRAMANIAN2003293}
V.~Venkatasubramanian, R.~Rengaswamy, K.~Yin, and S.~N. Kavuri, ``A review of
  process fault detection and diagnosis: Part i: Quantitative model-based
  methods,'' \emph{Computers \& Chemical Engineering}, vol.~27, no.~3, pp.
  293--311, 2003.

\bibitem{comp_methods}
M.~Thirumarimurugan, N.~Bagyalakshmi, and P.~Paarkavi, ``Comparison of fault
  detection and isolation methods: A review,'' in \emph{2016 10th International
  Conference on Intelligent Systems and Control (ISCO)}, 2016, pp. 1--6.

\bibitem{beard1971failure}
R.~V. Beard, ``Failure accomodation in linear systems through
  self-reorganization,'' Ph.D. dissertation, Massachusetts Institute of
  Technology, 1971.

\bibitem{jones1973failed}
H.~L. Jones, ``Failed detection in linear systems,'' Ph.D. dissertation,
  Massachusetts Institute of Technology, 1973.

\bibitem{massoumnia1986a}
M.-A. Massoumnia, ``A geometric approach to failure detection and
  identification in linear systems,'' Ph.D. dissertation, Massachusetts
  Institute of Technology, 1986.

\bibitem{massoumnia1986b}
------, ``A geometric approach to the synthesis of failure detection filters,''
  \emph{IEEE Transactions on Automatic Control}, vol.~31, no.~9, pp. 839--846,
  1986.

\bibitem{massoumnia1989failure}
M.-A. Massoumnia, G.~C. Verghese, and A.~S. Willsky, ``Failure detection and
  identification,'' \emph{IEEE Transactions on Automatic Control}, vol.~34,
  no.~3, pp. 316--321, 1989.

\bibitem{white1987detection}
J.~White and J.~Speyer, ``Detection filter design: Spectral theory and
  algorithms,'' \emph{IEEE Transactions on Automatic Control}, vol.~32, no.~7,
  pp. 593--603, 1987.

\bibitem{min1987detection}
P.~S. Min, ``Detection of incipient failures in dynamic systems,'' Ph.D.
  dissertation, University of Michigan, 1987.

\bibitem{clark1975}
R.~N. Clark, D.~C. Fosth, and V.~M. Walton, ``Detecting instrument malfunctions
  in control systems,'' \emph{IEEE Transactions on Aerospace and Electronic
  Systems}, vol. AES-11, no.~4, pp. 465--473, 1975.

\bibitem{EDWARDS2000541}
C.~Edwards, S.~K. Spurgeon, and R.~J. Patton, ``Sliding mode observers for
  fault detection and isolation,'' \emph{Automatica}, vol.~36, no.~4, pp.
  541--553, 2000.

\bibitem{YAN20071605}
X.-G. Yan and C.~Edwards, ``Nonlinear robust fault reconstruction and
  estimation using a sliding mode observer,'' \emph{Automatica}, vol.~43,
  no.~9, pp. 1605--1614, 2007.

\bibitem{tan2003sliding}
C.~P. Tan and C.~Edwards, ``Sliding mode observers for robust detection and
  reconstruction of actuator and sensor faults,'' \emph{International Journal
  of Robust and Nonlinear Control}, vol.~13, no.~5, pp. 443--463, 2003.

\bibitem{zhu2022}
F.~Zhu, Y.~Tang, and Z.~Wang, ``Interval-observer-based fault detection and
  isolation design for t-s fuzzy system based on zonotope analysis,''
  \emph{IEEE Transactions on Fuzzy Systems}, vol.~30, no.~4, pp. 945--955,
  2022.

\bibitem{xu2019conservatism}
F.~Xu, J.~Tan, X.~Wang, and B.~Liang, ``Conservatism comparison of set-based
  robust fault detection methods: Set-theoretic {UIO} and interval observer
  cases,'' \emph{Automatica}, vol. 105, pp. 307--313, 2019.

\bibitem{zhang2017event}
Z.-H. Zhang and G.-H. Yang, ``Event-triggered fault detection for a class of
  discrete-time linear systems using interval observers,'' \emph{ISA
  Transactions}, vol.~68, pp. 160--169, 2017.

\bibitem{su2020fault}
Q.~Su, Z.~Fan, T.~Lu, Y.~Long, and J.~Li, ``Fault detection for switched
  systems with all modes unstable based on interval observer,''
  \emph{Information Sciences}, vol. 517, pp. 167--182, 2020.

\bibitem{iqbal2019}
R.~Iqbal, T.~Maniak, F.~Doctor, and C.~Karyotis, ``Fault detection and
  isolation in industrial processes using deep learning approaches,''
  \emph{IEEE Transactions on Industrial Informatics}, vol.~15, no.~5, pp.
  3077--3084, 2019.

\bibitem{YANG201985}
J.~Yang, Y.~Guo, and W.~Zhao, ``Long short-term memory neural network based
  fault detection and isolation for electro-mechanical actuators,''
  \emph{Neurocomputing}, vol. 360, pp. 85--96, 2019.

\bibitem{farahani2020}
H.~S. Farahani, A.~Fatehi, and M.~A. Shoorehdeli, ``On the application of
  domain adversarial neural network to fault detection and isolation in power
  plants,'' in \emph{2020 19th IEEE International Conference on Machine
  Learning and Applications (ICMLA)}, 2020, pp. 1132--1138.

\bibitem{GUO2018155}
D.~Guo, M.~Zhong, H.~Ji, Y.~Liu, and R.~Yang, ``A hybrid feature model and deep
  learning based fault diagnosis for unmanned aerial vehicle sensors,''
  \emph{Neurocomputing}, vol. 319, pp. 155--163, 2018.

\bibitem{kazantzis1998}
N.~Kazantzis and C.~Kravaris, ``Nonlinear observer design using lyapunov’s
  auxiliary theorem,'' \emph{Systems \& Control Letters}, vol.~34, no.~5, pp.
  241--247, 1998.

\bibitem{andrieu2006existence}
V.~Andrieu and L.~Praly, ``On the existence of a
  {Kazantzis--Kravaris/Luenberger} observer,'' \emph{SIAM Journal on Control
  and Optimization}, vol.~45, no.~2, pp. 432--456, 2006.

\bibitem{milanese1996}
M.~Milanese, J.~Norton, H.~Piet-Lahanier, and {\'E}.~Walter, \emph{Bounding
  Approaches to System Identification}.\hskip 1em plus 0.5em minus 0.4em\relax
  Springer New York, NY, 1996.

\bibitem{mania2022}
H.~Mania, M.~I. Jordan, and B.~Recht, ``Active learning for nonlinear system
  identification with guarantees.'' \emph{Journal of Machine Learning
  Research}, vol.~23, pp. 1--30, 2022.

\bibitem{niazi2022}
M.~U.~B. Niazi, J.~Cao, X.~Sun, A.~Das, and K.~H. Johansson, ``Learning-based
  design of luenberger observers for autonomous nonlinear systems,'' in
  \emph{2023 American Control Conference (ACC)}, 2023, pp. 3048--3055.

\bibitem{fonod2014}
R.~Fonod, D.~Henry, C.~Charbonnel, and E.~Bomschlegl, ``A class of nonlinear
  unknown input observer for fault diagnosis: Application to fault tolerant
  control of an autonomous spacecraft,'' in \emph{2014 UKACC International
  Conference on Control}, 2014, pp. 13--18.

\bibitem{vemuri1997neural}
A.~T. Vemuri and M.~M. Polycarpou, ``Neural-network-based robust fault
  diagnosis in robotic systems,'' \emph{IEEE Transactions on Neural Networks},
  vol.~8, no.~6, pp. 1410--1420, 1997.

\bibitem{dorfler2014}
F.~D{\"o}rfler and F.~Bullo, ``Synchronization in complex networks of phase
  oscillators: A survey,'' \emph{Automatica}, vol.~50, no.~6, pp. 1539--1564,
  2014.

\bibitem{bernard2022}
P.~Bernard, V.~Andrieu, and D.~Astolfi, ``Observer design for continuous-time
  dynamical systems,'' \emph{Annual Reviews in Control}, vol.~53, pp. 224--248,
  2022.

\bibitem{sontag2013}
E.~D. Sontag, \emph{Mathematical Control Theory: Deterministic Finite
  Dimensional Systems}, 2nd~ed.\hskip 1em plus 0.5em minus 0.4em\relax Springer
  New York, NY, 2013, vol.~6.

\end{thebibliography}

\end{document}